\documentclass{amsart}

\usepackage{amsfonts, amsmath, amsthm,amssymb}
\usepackage[dvips]{graphicx}
\usepackage{amscd}
\makeatletter
\@addtoreset{equation}{section}

\makeatother

\title[A Fr\'echet topology on measured laminations]{A Fr\'echet topology on measured laminations and
Earthquakes in the hyperbolic plane}

\author{Hideki Miyachi and Dragomir \v{S}ari\'c}
\thanks{The first author is partially supported by
Grant-in-Aid for Scientific Research (C) 21540177}
\thanks{The second author was partially supported by
PSC-CUNY grant PSCREG-40-136.} \keywords{} \subjclass[2000]{}
\date{}

\address{Department of Mathematics,
Graduate School of Science, Osaka University,
Machikaneyama 1-1, Toyonaka, Osaka, 560-0043, Japan}
\email{miyachi@math.sci.osaka-u.ac.jp}

\address{Department of Mathematics,
CUNY Queens College and Graduate Center, 65-30 Kissena Blvd.,
Flushing, NY 11367} \email{dragomir.saric@qc.cuny.edu}

\newtheorem{theorem}{Theorem}
\newtheorem{proposition}{Proposition}[section]
\newtheorem{lemma}{Lemma}[section]

\newtheorem{remark}{Remark}[section]

\newcommand{\hol}{{\rm H\ddot{o}l}}
\newcommand{\geodesic}[2]{\lceil#1,#2\rceil}

\newcommand{\zyg}{\mathcal{Z}}

\begin{document}

\begin{abstract}
We prove that the bijective correspondence between the space of
bounded measured laminations $ML_b(\mathbb{H})$ and the universal
Teichm\"uller space $T(\mathbb{H})$ given by $\lambda\mapsto
E^{\lambda}|_{S^1}$ is a homeomorphism for the Fr\'echet topology on
$ML_b(\mathbb{H})$ and the Teichm\"uller topology on
$T(\mathbb{H})$, where $E^{\lambda}$ is an earthquake with
earthquake measure $\lambda$. A corollary is that earthquakes with
discrete earthquake measures are dense in $T(\mathbb{H})$. We also
establish infinitesimal versions of the above results.
\end{abstract}

\maketitle

\section{Introduction}

A Riemann surface is said to be \emph{hyperbolic} if its universal
covering is the hyperbolic plane $\mathbb{H}$. \footnote{We are
particularly interested in the geometrically infinite hyperbolic
Riemann surfaces, e.g. the hyperbolic plane $\mathbb{H}$, an
infinite genus surface, a surface with an interval of ideal boundary
points. All these surfaces have infinite hyperbolic area.} A
quasiconformal map between two hyperbolic Riemann surfaces lifts to
a quasiconformal map between their universal coverings, which are
identified with the hyperbolic plane $\mathbb{H}$. This map
continuously extends to a quasisymmetric map of the boundary
$\partial\mathbb{H}$ of the hyperbolic plane, which is in turn
identified with the unit circle $S^1$. The homotopy class of a
quasiconformal map between two Riemann surfaces is uniquely
determined by the quasisymmetric map of $S^1$, and this induces a
natural complex analytic embedding of the Teichm\"uller space of any
hyperbolic Riemann surface into the Teichm\"uller space
$T(\mathbb{H})$ of the hyperbolic plane $\mathbb{H}$, called the
\emph{universal Teichm\"uller space}.

The universal Teichm\"uller space $T(\mathbb{H})$ is the space of
all quasisymmetric maps of the unit circle $S^1$ modulo
post-composition by M\"obius maps which preserve $\mathbb{H}$. It is
an infinite-dimensional complex Banach manifold which contains other
interesting spaces of circle maps. We study $T(\mathbb{H})$ by the
use of the hyperbolic geometry of $\mathbb{H}$. Our main objects are
earthquakes in the hyperbolic plane $\mathbb{H}$ and H\"older
distributions on the space $\mathcal{G}$ of geodesics of the
hyperbolic plane $\mathbb{H}$.

Earthquake maps in the hyperbolic plane $\mathbb{H}$ (and on any
hyperbolic Riemann surface) were introduced by Thurston
\cite{Thurston}. An earthquake in the hyperbolic plane is a
bijective map $E:\mathbb{H}\to\mathbb{H}$ which is \emph{supported}
on a geodesic lamination $\mathcal{L}$ in $\mathbb{H}$ in the sense
that it is a hyperbolic isometry on each \emph{stratum} (i.e. a leaf
of $\mathcal{L}$ or a component of $\mathbb{H}\setminus\mathcal{L}$)
of $\mathcal{L}$, and which (relatively) translates to the left
points of different strata of $\mathcal{L}$. An earthquake
$E:\mathbb{H}\to\mathbb{H}$ continuously extends to a homeomorphism
of $S^1$ and it induces a transverse Borel measure to its support
lamination $\mathcal{L}$, called the \emph{earthquake measure}. The
earthquake measure of $E$ measures the amount of the relative
movement to the left by $E$. An earthquake measure $\lambda$
uniquely determines earthquake $E^{\lambda}:\mathbb{H}\to\mathbb{H}$
up to post-composition by M\"obius maps.

Thurston \cite{Thurston} showed that any homeomorphism of the unit
circle $S^1$ is obtained as the continuous extension of an
earthquake in $\mathbb{H}$ to its boundary $S^1$. In other words,
any homeomorphism of $S^1$ can be geometrically constructed as the
continuous extension to the boundary $S^1$ of a piecewise isometry
of $\mathbb{H}$ which moves strata of its support geodesic
lamination to the left by the amount given by a transverse Borel
measure to the lamination. However, the relationship between
homeomorphisms and earthquake measures of the earthquakes inducing
them is not a simple one. This paper is mainly concerned with the
dependence of the earthquake measures on homeomorphisms of $S^1$.

A measured lamination $\lambda$ is said to be \emph{bounded} if
$$
\sup_I\lambda (I)<\infty
$$
where the supremum is over all geodesic arcs $I$ of unit length that
transversely intersect the support of $\lambda$. Then a
homeomorphism is quasisymmetric if and only if
$h=E^{\lambda}|_{S^1}$ for a bounded earthquake measure $\lambda$
(see \cite{GHL}, \cite{Saric4} and \cite{Saric2}).

We denote by $ML_b(\mathbb{H} )$ the space of all bounded measured
laminations. The above statement gives a well-defined
\emph{earthquake measure map}
$$
\mathcal{EM}:T(\mathbb{H})\to ML_b(\mathbb{H})
$$
by $\mathcal{EM}([h])=\lambda$, where quasisymmetric map $h$ is
obtained by continuously extending to $S^1$ earthquake $E^{\lambda}$
with earthquake measure $\lambda$. The earthquake measure map is a
bijection by the above. Our main result establishes a natural
topology on $ML_b(\mathbb{H})$ for which $\mathcal{EM}$ is a
homeomorphism.

Each oriented geodesic in $\mathbb{H}$ is uniquely determined by the
pair of its endpoints on $S^1$, the initial point and the terminal
point. Then the space $\mathcal{G}$ of unoriented geodesics in
$\mathbb{H}$ is isomorphic to $(S^1\times S^1\setminus diag)/\sim$,
where $(a,b)\sim (b,a)$ and $diag=\{ (a,a)|a\in S^1\}$. We fix an
angle metric on $S^1$ and obtain an induced metric $d$ on
$\mathcal{G}$. Let $H\ddot{o}l_0$ be the space of all H\"older
continuous functions $\varphi :\mathcal{G}\to\mathbb{R}$ with
compact support. For $0<\nu\leq 1$, let $H\ddot{o}l_0^{\nu}$ be the
space of all $\nu$-H\"older continuous functions $\varphi
:\mathcal{G}\to\mathbb{R}$ with compact support. Let
$Q^{*}=[([-i,1]\times [i,-1])/\sim ]\subset\mathcal{G}$.

Let $test(\nu )$ be the space of pairs $(\varphi ,Q)$ with the
following properties. The function $\varphi
:\mathcal{G}\to\mathbb{R}$ is $\nu$-H\"older continuous and its
support is contained in $Q=([a,b]\times [c,d])/\sim$. The closed
arcs $[a,b],[c,d]\subset S^1$ are disjoint and the \emph{Liouville
measure} $L(Q):=\log\frac{(a-c)(b-d)}{(a-d)(b-c)}$  of $Q$ equals
$\log 2$. If $\gamma_Q:Q^{*}\mapsto Q$ is a M\"obius map, then
$$
\|\varphi\circ\gamma_Q\|_{\nu}<\infty
$$
where $\|\varphi\|_{\nu}$ is the $\nu$-H\"older norm of $\varphi$
(cf. \S 2.4).

The space $\mathcal{H}$ of H\"older distributions consists of all
linear functionals $W:H\ddot{o}l_0\to\mathbb{R}$ such that
$$
\|W\|_{\nu}:=\sup_{(\varphi ,Q)} |W(\varphi )|<\infty
$$
for each $\nu$, $0<\nu\leq 1$, where the supremum is over all
$(\varphi , Q)\in test(\nu )$. The family of $\nu$-norms on
$\mathcal{H}$ induces a Fr\'echet structure on $\mathcal{H}$. The
space of H\"older distributions for closed surfaces is introduced by
Bonahon \cite{Bonahon1}, and generalized in the above form for
geometrically infinite surfaces \cite{Saric3}. The Liouville map
$\mathcal{L}:T(\mathbb{H})\to \mathcal{H}$ given by the pull-backs
of the Liouville measure is an analytic homeomorphism onto its image
(cf. \cite{Bonahon2}, \cite{Saric5}, \cite{Otal}). Bonahon
\cite{Bonahon2} defined Thurston boundary to Teichm\"uller spaces of
closed surfaces using the Liouville map and his construction extends
to geometrically infinite surfaces \cite{Saric5}.

Our main result makes a connection between the Fr\'echet topology on
$\mathcal{H}$ and earthquake maps in the hyperbolic plane. Namely,
we show that the induced Fr\'echet topology on
$ML_b(\mathbb{H})\subset\mathcal{H}$ is capturing the subtleties of
the Teichm\"uller topology on $T(\mathbb{H})$ and the earthquake
maps in the hyperbolic plane $\mathbb{H}$.

\begin{theorem}[Earthquake measure map is a homeomorphism]
\label{thm:EMisHomeo} The earthquake measure map
$$
\mathcal{EM}:T(\mathbb{H})\to ML_b(\mathbb{H})
$$
is a homeomorphism for the Teichm\"uller topology of $T(\mathbb{H})$
and the Fr\'echet topology on $ML_b(\mathbb{H})$.
\end{theorem}

The above theorem also holds for any geometrically infinite Riemann
surfaces by simply noting that a quasisymmetric map which is
invariant under a Fuchsian group is induced by an earthquake whose
earthquake measure is invariant under the same Fuchsian group. In
the case of a closed hyperbolic surface $S$, Kerckhoff \cite{Ker}
showed that the earthquake measure map is a homeomorphism for the
weak* topology on $ML(S)$. Using the techniques in the paper, it is
easy to prove that
$\mathcal{EM}:M\ddot{o}b(\mathbb{H})/Homeo(S^1)\to ML(\mathbb{H})$
is a homeomorphism for the topology of pointwise convergence on the
space of homeomorphisms $Homeo(S^1)$ of $S^1$ and the weak* topology
on the (not necessarily bounded) measured laminations
$ML(\mathbb{H})$ of $\mathbb{H}$, where $M\ddot{o}b(\mathbb{H})$ are
M\"obius maps that preserve $\mathbb{H}$. We note that the weak*
topology on $ML_b(\mathbb{H})$ is strictly weaker than the Fr\'echet
topology.

\begin{figure}
\includegraphics[height=5cm]{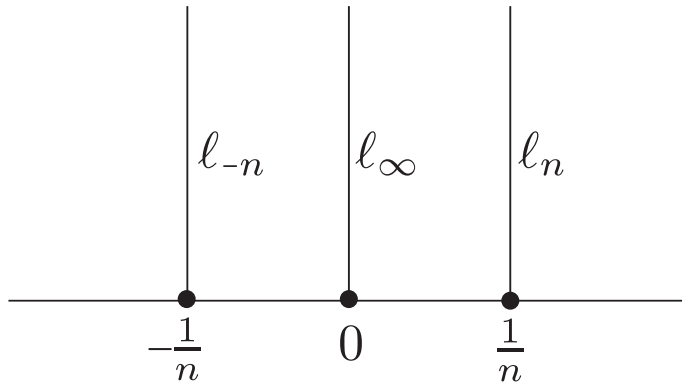}
\caption{$\lambda_n\nrightarrow\lambda$ in the Fr\'echet topology.}
\end{figure}

To illustrate the difference between the weak* topology and the
Fr\'echet topology on $ML_b(\mathbb{H})$ we consider the following
example. Identify the hyperbolic plane $\mathbb{H}$ with the upper
half-plane and its boundary $\partial\mathbb{H}$ with
$\hat{\mathbb{R}}=\mathbb{R}\cup\{\infty\}$. Let $l=(0,\infty
)/\sim$ and $l_n=(\frac{1}{n},\infty )/\sim$ be geodesics in
$\mathbb{H}$. Let $\delta_l$ and $\delta_{l_n}$ denote the Dirac
measures on $\mathcal{G}$ with supports $l$ and $l_n$, respectively.
Then $\frac{\delta_{l_n}+\delta_{l_{-n}}}{2}$ converges in the weak*
topology to $\delta_l$ as $n\to\infty$, but it does \emph{not}
converge in the Fr\'echet topology (Figure \ref{}). See \S 5 for
further discussion and examples.

An earthquake is said to be \emph{finite} if its earthquake measure
has finite support in $\mathcal{G}$. Thurston \cite{Thurston} proved
that the graph of any earthquake $E:\mathbb{H}\to\mathbb{H}$ is
approximated by the graphs of finite earthquakes. Gardiner-Hu-Lakic
\cite{GHL} proved that each monotone map from an $n$-tuple of points
in $S^1$ into $S^1$ can be realized by a finite earthquake whose
support geodesics are in the $n$-tuple (finite earthquake theorem).
We say that an earthquake is \emph{discrete} if the support of its
earthquake measure is a discrete subset of $\mathcal{G}$. Next to
finite earthquakes, discrete earthquakes are the simplest possible
earthquakes and, by definition, finite earthquakes are discrete. We
prove that each earthquake $E$ can be approximated by a sequence of
discrete earthquakes $E_n$ in the sense that $E|_{S^1}\to
E_n|_{S^1}$ in the Teichm\"uller topology as $n\to\infty$. Theorem
below is a direct consequence of Theorem \ref{thm:approximation}
(cf. \S 7.2) and Theorem \ref{thm:EMisHomeo}.

\begin{theorem}[Countable Earthquake Theorem]
\label{thm:CountableEarthquakeTheorem} Let $ML_{b}^{disc}$ be the
set of all bounded measured laminations whose supports are discrete
subsets of $\mathcal{G}$. Then the set
$$
\{ [E^{\lambda}|_{S^1}]:\lambda\in ML_b^{disc}\}
$$
is a dense subset of $T(\mathbb{H})$ in the Teichm\"uller topology.
\end{theorem}

We prove analogous statements for the Zygmund vector fields and the
infinitesimal earthquakes. Let $V$ be a vector field on $S^1$ and
let $Q=([a,b]\times [c,d])/\sim$, called a \emph{box of geodesics},
be a subset of $\mathcal{G}$ such that $[a,b]\cap [c,d]=\emptyset$.
Define
$$V[Q]:=\frac{V(a)-V(c)}{a-c}-\frac{V(a)-V(d)}{a-d}+\frac{V(b)-V(d)}{b-d}
-\frac{V(c)-V(d)}{c-d}.
$$
The \emph{cross-ratio norm} $\| V\|_{cr}$ of a vector field $V$ is
defined by
$$
\| V\|_{cr}:=\sup_Q V[Q],
$$
where the supremum is over all boxes of geodesics $Q=([a,b]\times
[c,d])/\sim$ with $L(Q)=\log 2$. A vector field $V$ on $S^1$ is
\emph{Zygmund bounded} if its cross-ration norm $\| V\|_{cr}$ is
finite. Let $\mathcal{Z}(S^1)$ be the vector space of all Zygmund
bounded vector fields on $S^1$ modulo the closed subspace of
quadratic polynomials. (Note that quadratic polynomials are
infinitesimal deformations of the paths of M\"obius maps.)

A vector field $V$ on $S^1$ is Zygmund bounded if and only if there
exists a differentiable path of quasisymmetric maps $t\mapsto h_t$,
for $|t|<\epsilon$ with $\epsilon >0$, such that $h_0=id$ and
$\frac{d}{dt}h_t|_{t=0}=V$ (see \cite{GL}). Given $\lambda\in
ML_b(\mathbb{H})$, the path $t\mapsto E^{t\lambda}|_{S^1}$ is
differentiable. Its derivative at $t=0$ is a Zygmund bounded vector
field, called the \emph{infinitesimal earthquake}, and we denote it
by
$$
\dot{E}^{\lambda}|_{S^1}:=\frac{d}{dt}(E^{t\lambda}|_{S^1})|_{t=0}.
$$
Gardiner \cite{gardiner} proved that each Zygmund bounded vector
field arises as an infinitesimal earthquake and he also established
the formula (see also \S 9)
$$
\dot{E}^{\lambda}|_{S^1}=\int_{\mathcal{G}}\dot{E}^\lambda_{\ell}d\lambda
(l),
$$
where $\dot{E}^\lambda_{\ell}(z)=\frac{(z-a)(z-b)}{a-b}$ for $z\in
S^1$ with $a$ and $b$ the endpoints of $\ell$ such that the triple
$(a,z,b)$ has positive orientation on $S^1$.

The \emph{infinitesimal earthquake measure map}
$$
\dot{\mathcal{EM}}:ML_b(\mathbb{H})\to\mathcal{Z}(S^1)
$$
defined by
$$
\dot{\mathcal{EM}}:\lambda\mapsto\dot{E}^{\lambda}|_{S^1}
$$
is a bijection. We prove that the Fr\'echet topology on
$ML_b(\mathbb{H})$ makes $\dot{\mathcal{EM}}$ into a homeomorphisms
analogous to the case of quasisymmetric maps.

\begin{theorem}[Fr\'echet and Zygmund] \label{thm:Frechet_and_Zygmund}
Let $\mathcal{ML}_b(\mathbb{D})$ be given the Fr\'echet topology and
$\mathcal{Z}(S^1)$ be given the cross-ratio norm topology. Then, the
infinitesimal earthquake measure map
$$
\dot{\mathcal{EM}}:ML_b(\mathbb{H})\to\mathcal{Z}(S^1)
$$
is a homeomorphism.
\end{theorem}

An infinitesimal version of the countable earthquake theorem
immediately follows from Theorem \ref{thm:approximation} in \S 7 and
Theorem \ref{thm:Frechet_and_Zygmund}.

\section{Measured laminations and H\"older distributions}
\subsection{Space of geodesics}
\label{subsec:space_of_geodesics} Let $\mathbb{D}$ be the unit disk
model of the hyperbolic plane $\mathbb{H}$. The unit circle $S^1$ is
identified with the set of ideal boundary points
$\partial\mathbb{D}$ of the hyperbolic plane. Fix
$z_0\in\mathbb{D}$. Define the distance between $z_1,z_2\in S^1$ to
be smaller angle between the geodesic rays connecting $z_0$ with
$z_1$ and $z_2$, respectively. This gives an angle metric on $S^1$
which depends on $z_0$. By varying $z_0\in\mathbb{D}$ we obtain a
biLipschitz class of metrics on $S^1$.

A complete oriented geodesic $g$ in $\mathbb{D}$ is uniquely
determined by an ordered pair of its distinct ideal endpoints on
$S^1$, the initial and the terminal point of $g$. Conversely, given
an ordered pair of points on $S^1$, there is a unique oriented
hyperbolic geodesic with its initial endpoint being the first point
and its terminal endpoint being the second point of the pair. Thus
the space $\tilde{\mathcal{G}}$ of all oriented geodesics on
$\mathbb{D}$ is naturally identified with $S^1\times S^1\setminus
diag$. Let $\mathcal{G}$ be the set of all unoriented complete
hyperbolic geodesic on $\mathbb{D}$. The set $\mathcal{G}$ is
identified with $(S^1\times S^1\setminus diag)/\sim$, where the
equivalence is defined by $(a,b)\sim (b,a)$ and $diag$ is the
diagonal set of the product. We denote by $\geodesic{a}{b}$ the
equivalence class of $(a,b)\in S^1\times S^1\setminus diag$. An
angle metric $d_{z_0}$ on $S^1$ with respect to $z_0\in\mathbb{D}$
induces a metric $\bar{d}_{z_0}$ on $\mathcal{G}$ as follows. Let
$\geodesic{a}{b},\geodesic{c}{d}\in\mathcal{G}$. Define
$\bar{d}_{z_0}(\geodesic{a}{b},\geodesic{c}{d})=\min\{ \max\{
d_{z_0}(a,c),d_{z_0}(b,d)\} , \max\{ d_{z_0}(a,d),d_{z_0}(b,c)\}\}$.
The set of geodesics $\mathcal{G}$ has a biLipschitz class of
metrics obtained by varying $z_0\in\mathbb{D}$.

A quasiconformal map $f:\mathbb{D}\to \mathbb{D}$ continuously
extends to a quasisymmetric map $h:S^1\to S^1$. Mori's theorem
implies that $h$ is a H\"older continuous homeomorphism of $S^1$
whose H\"older constant depends only on the maximal dilatation of
$f$. Thus a quasisymmetric mapping of $S^1$ also induces a H\"older
continuous homeomorphism of $\mathcal{G}$ for the angle metric
$\bar{d}_{z_0}$. Since each quasisymmetric map induces a
biholomorphic isometry of the universal Teichm\"uller space, it is
natural to work with the class of H\"older equivalent metrics to the
metric $\bar{d}_{z_0}$. Recall that a metric $d$ is H\"older
equivalent to $\bar{d}_{z_0}$ if there exist $C\geq 1$ and
$0<\nu\leq 1$ such that
$$
\frac{1}{C}d(\geodesic{x}{y},\geodesic{x_1}{y_1})^{\frac{1}{\nu}}\leq\bar{d}_{z_0}
(\geodesic{x}{y},\geodesic{x_1}{y_1})\leq C
d(\geodesic{x}{y},\geodesic{x_1}{y_1})^{\nu}.
$$

\subsection{Measured laminations}
A \emph{geodesic lamination} $\mathcal{L}$ is a closed subset of
$\mathbb{D}$ together with a foliation by disjoint complete
geodesics. We recall that the information of the foliation of the
closed subset is necessary for the definition of a geodesic
lamination in $\mathbb{D}$. For example, the hyperbolic plane can be
foliated by complete hyperbolic geodesics in infinitely many
different ways and each different foliation determines a different
geodesic lamination. Equivalently, a geodesic lamination
$\mathcal{L}$ is a closed subset of $\mathcal{G}$ such that no two
geodesics in $\mathcal{L}$ intersect in $\mathbb{D}$ (they can have
common ideal endpoints).

Each complete geodesic in $\mathcal{L}$ is called a \emph{leaf} of
$\mathcal{L}$. A \emph{stratum} of $\mathcal{L}$ is either a
geodesic of $\mathcal{L}$ or a component of the complement of
$\mathcal{L}$ in $\mathbb{D}$.

A \emph{measured lamination} $\lambda$ is a positive, locally
finite, Borel measure on the space of geodesics $\mathcal{G}$ whose
support $|\lambda |$ is a geodesic lamination. Each measured
lamination $\lambda$ induces a \emph{transverse measure} to its
support $|\lambda |$, namely an assignment of a positive, Borel
measure to each closed finite hyperbolic arc $I$ in $\mathbb{D}$
whose support is $I\cap |\lambda |$ and which is invariant under
homotopies which preserve the strata of $|\lambda |$. More
precisely, the $\lambda$-mass of an arc $I$, denoted by $\lambda
(I)$, is the $\lambda$-measure of the set of geodesics in
$\mathcal{G}$ which intersect $I$. Conversely, a transverse measure
to a geodesic lamination $\mathcal{L}$ determines a unique measured
lamination $\lambda$ whose support is $\mathcal{L}=|\lambda |$. For
this correspondence we refer the reader to \S1 of \cite{Bonahon1}. A
measured lamination $\lambda$ is \emph{bounded} if the
\emph{Thurston's norm}
$$
\|\lambda\|_{Th}=\sup_{I}\lambda(I)
$$
is finite, where $I$ runs over all geodesic arcs in $\mathbb{D}$
with unit length. Let $\mathcal{ML}_b(\mathbb{D})$ be the set of
bounded measured laminations on $\mathbb{D}$. When the support of a
measured lamination $\lambda$ consists of one geodesic, we say that
$\lambda$ is an \emph{elementary measured lamination}.

M\"obius transformations act isometrically on the set of bounded
measured laminations by the \emph{pull-backs} as follows. Let
$\gamma\in {\rm M\ddot{o}b}(\mathbb{D})$ and $\lambda$ a measured
lamination. We define  $\gamma^*\lambda$ as the measured lamination
with support $\gamma^{-1}(|\lambda|)$ and the transverse measure
$\lambda\circ \gamma$, where $(\lambda\circ
\gamma)(I)=\lambda(\gamma(I))$ for all geodesic arcs $I$. Clearly,
$$
\|\gamma^*\lambda\|_{Th}=\|\lambda\|_{Th}
$$
holds for any measured lamination $\lambda$, and hence ${\rm
M\ddot{o}b}(\mathbb{D})$ acts by isometry on
$\mathcal{ML}_b(\mathbb{D})$.

\subsection{Boxes and the Liouville measure}
\label{subsec:Box_Liouville_measure} The \emph{cross ratio} of a
quadruple $(a,b,c,d)$ is given by
$cr(a,b,c,d)=\frac{(a-c)(b-d)}{(a-d)(b-c)}$. A \emph{box of
geodesics} $Q$ in $\mathcal{G}$ is the quotient under the
equivalence $\sim$ of the product $[a,b]\times [c,d]$ of two
disjoint closed arcs in $S^1$, where $[a,b]$ (resp. $[c,d]$) is the
arc in $S^1$ from $a$ (resp. $c$) to $b$ (resp. $d$) for the
orientation of $S^1$. We will write somewhat incorrectly
$Q=[a,b]\times [c,d]$ instead of a more correct $Q=([a,b]\times
[c,d])/\sim$. The \emph{Liouville measure} $L$ is a canonical,
non-trivial, M\"obius group invariant Borel measure on $\mathcal{G}$
defined by
$$
L(Q)=
\left|\log |cr(a,b,c,d)|\right|=
\left|\log
\left|\frac{(a-c)(b-d)}{(a-d)(b-c)}\right|
\right|
$$
for all boxes $Q=[a,b]\times [c,d]$. The Liouville measure is unique
up to scaling. The infinitesimal form of the Liouville measure on
$\mathcal{G}=(S^1\times S^1\setminus diag)\sim$ is given by (see
\cite{Bonahon2})
$$ dL=\frac{d\alpha
d\beta}{|e^{i\alpha}-e^{i\beta}|^2}.$$
For instance, when we
consider the upper half-plane model $\mathbb{H}$ of the hyperbolic
plane instead of $\mathbb{D}$ and let $Q=[-1,1]\times [e^D,-e^D]$,
the Liouville measure of $Q$ is
\begin{equation} \label{eq:liouville_measure}
L(Q)=-2\log \tanh\frac{D}{2}.
\end{equation}
Thus, for a general square $Q=[a,b]\times [c,d]$, the Liouville
measure $L(Q)$ is inversely related to the hyperbolic distance
between the geodesics $\geodesic{a}{b}$ and $\geodesic{c}{d}$.
Furthermore, a square $Q=[a,b]\times[c,d]$ satisfies $L(Q)=\log 2$
if and only if the distance $D$ between $\geodesic{a}{b}$ and
$\geodesic{c}{d}$ satisfies $e^D=\omega_0$ $(=(1+\sqrt{2})^2)$ if
and only if the distance between $\geodesic{a}{b}$ and
$\geodesic{c}{d}$ equals the distance between $\geodesic{a}{d}$ and
$\geodesic{b}{c}$. A short computation shows that the box
$Q=[-1,1]\times [3+2\sqrt{2},-(3+\sqrt{2})]\subset
(\hat{\mathbb{R}}\times\hat{{\mathbb{R}}}\setminus diag)\sim$ has
the Liouville measure $\log 2$.

We again consider the unit disk model $\mathbb{D}$ of the hyperbolic
plane and define $Q^*=[-i,1]\times [i,1]$. Let
$\ell_{Q^*}=\geodesic{e^{-\pi/4}}{e^{3\pi/4}}\in Q^*$. Let $Q$ be a
box with $L(Q)=\log 2$ and $\gamma_Q$ a M\"obius transformation of
$\mathbb{D}$ with $\gamma_Q(Q^*)=Q$. The geodesic
$\ell_Q:=\gamma_Q(\ell_{Q^*})$ is called the \emph{center} of the
box $Q$.

\subsection{H\"older distribution} Let $d_0$ be the angle metric on $S^1$
with respect to the origin $0\in\mathbb{D}$. Let $d$ be the metric
on $\mathcal{G}$ induced by $d_0$ as in
\S\ref{subsec:space_of_geodesics}. A H\"older continuous function
$\varphi :\mathcal{G}\to\mathbb{R}$ with respect to the fixed metric
$d$ on $\mathcal{G}$ is H\"older continuous for the whole class of
H\"older equivalent metrics to the metric $d$. Unless otherwise
stated, all the constructions that follow are with respect to the
fixed metric $d$ on $\mathcal{G}$.

The space $\hol_0$ consists of all H\"older continuous function
$\varphi:\mathcal{G}\to \mathbb{R}$ with compact support, where
$\mathcal{G}$ is equipped with the fixed metric $d$. Let $0<\nu\leq
1$. For a $\nu$-H\"older continuous function $\varphi$ on
$\mathcal{G}$, we define its $\nu$-norm by
$$
\|\varphi\|_\nu= \max \left\{ \max |\varphi(\geodesic{x}{y})|, \sup
\frac{|\varphi(\geodesic{x}{y})-\varphi(\geodesic{x_1}{y_1})|}{d(\geodesic{x}{y},\geodesic{x_1}{y_1})^\nu}
\right\},
$$
where the maximum inside the brackets is over all
$\geodesic{x}{y}\in \mathcal{G}$ and where the supremum is over all
distinct $\geodesic{x}{y},\geodesic{x_1}{y_1}\in \mathcal{G}$. Let
us denote by $\hol_0^\nu$ the space all $\nu$-H\"older continuous
functions on $\mathcal{G}$ with compact support. Then,
$\hol_0=\cup_{0<\nu\le 1} \hol_0^\nu$.

A $\nu$-\emph{test function} is a pair $(\varphi,Q)$, where $Q$ is a
box of geodesics and $\varphi$ is a H\"older continuous function
such that $L(Q)=\log 2$, ${\rm supp}(\varphi)\subset Q$ and
$\|\varphi\circ \gamma_Q\|_\nu\le 1$. Recall that $\gamma_Q$ is a
unique M\"obius mapping which maps $Q^{*}=[-i,1]\times [i,-1]$ onto
$Q$. We denote by ${\rm test}(\nu)$ the set of $\nu$-test functions.

A $\nu$-\emph{H\"older distribution} is a linear functional $W$ on
$\hol_0^{\nu}$ such that
$$
\|W\|_\nu:=\sup\{|W(\varphi)|\mid (\varphi,Q)\in {\rm
test}(\nu)\}<\infty .
$$
A \emph{H\"older distribution} is a linear functional $W$ on
$\hol_0$ such that
$$
\| W\|_{\nu}<\infty
$$
for all $0<\nu\le 1$. In general, the $\nu$-H\"older norms $\|
W\|_{\nu}$ of a fixed H\"older distribution $W$ can increase without
a bound as $\nu\to 0$. Let $\mathcal{H}^\nu$ be the set of all
linear functionals $W$ on $\hol_0$ with $\|W\|_\nu<\infty$. Then
$\mathcal{H}^{\nu}$ is a Banach space for the $\nu$-norm
$\|\cdot\|_{\nu}$. The space $\mathcal{H}$ of all H\"older
distributions is equal to $\cap_{0<\nu\le 1}\mathcal{H}^\nu$. Each
$\|\cdot\|_{\nu}$ is a norm (i.e. is non-degenerate) on
$\mathcal{H}$, but $(\mathcal{H},\|\cdot\|_{\nu})$ is not a complete
space. The family of $\nu$-norms makes $\mathcal{H}$ into a
Fr\'echet space. Note that $\mathcal{H}$ is invariant under
quasisymmetric changes of coordinates on $S^1$ because
quasisymmetric maps are H\"older continuous, while each
$\mathcal{H}^{\nu}$ is not invariant. For more details, see
\cite{Saric3}.

\medskip
\paragraph{{\bf Special Test Functions}}
For the later use, we shall define a special test function
$(\psi_{0,\nu},Q^*)$ ($0<\nu\le 1$) as follows. Let
\begin{equation} \label{eq:small_square}
Q^*_0=[\omega_1^{13},\omega_1^{15}]\times [\omega_1^5,\omega_1^7]
\end{equation}
where $\omega_1=e^{i\pi/8}$ is a $16$-th root of unity. We now fix a
$C^\infty$ function $\varphi_0$ on $\mathcal{G}$ with the properties
that $\varphi_0\equiv 1$ on $Q^*_0$, $0\le \varphi_0\le 1$ and ${\rm
supp}(\varphi_0)\subset Q^*$. Since $\varphi_0$ is a Lipschitz
function and
\begin{equation*} \label{eq:Holder_comparizon}
\|\varphi_0\|_\nu\le (\pi/2)^{1-\nu}\|\varphi_0\|_1
\end{equation*}
for all $\nu$ with $0<\nu\le 1$ (cf. the equation (8) in \cite{Saric3}),
we have
\begin{equation} \label{eq:special_test_function}
(\psi_{0,\nu},Q^*):=(((\pi/2)^{1-\nu}\|\varphi_0\|_1)^{-1}\varphi_0,Q^*)
\in {\rm test}(\nu).
\end{equation}
Notice that $L(Q_0^*)=\log(4/(\sqrt{2}+2))$.

\subsection{Bounded measured laminations as H\"older distributions}
\label{subsec:bounded_measured_laminations_Holder} A \emph{Radon
measure} on a topological space is a locally finite Borel measure
with the inner regularity. It is known that any locally finite Borel
measure on a Suslin space (for instance, a separable and complete
metrizable space) is a Radon measure (cf. Theorem 11 of Chapter II
in \cite{Schwartz}).

\subsubsection{Weak* convergence}
We say that a sequence $\{\lambda_n\}_{n=1}^\infty$ of Borel
measures on $\mathcal{G}$ \emph{converges in the weak* topology} to
a Borel measure $\lambda$ if for all continuous function $f$ with
compact support on $\mathcal{G}$, it holds
$$
\lim_{n\to \infty}\int_{\mathcal{G}}f\,d\lambda_n=
\int_{\mathcal{G}}f\,d\lambda.
$$
(This convergence is sometimes called the \emph{vague convergence},
but we call it the weak* convergence here.)

\subsubsection{Measures of squares}
The following lemma is well-known.
However we give a proof for readers convenience.

\begin{lemma}[Comparison with Thurston norm] \label{lem:Thurston_norm_and_Holder_distribution}
There is a universal constant $C_0$ such that
for any measured lamination $\lambda$,
we have
$$
\frac{1}{C_0}\|\lambda\|_{Th}\le
\sup_Q\lambda(Q)
\le
\|\lambda\|_{Th},
$$
where
the supremum is taken over all boxes $Q$ with $L(Q)=\log 2$.
\end{lemma}

\begin{proof}
Let $I$ be a geodesic arc in $\mathbb{D}$ of the unit length which
intersects transversely a leaf $\ell$ of $\lambda$. Since the
support $|\lambda|$ consists of disjoint geodesics, there is a
universal constant $L_0$ with the following property: Let $J$ be a
geodesic arc in $\mathbb{D}$ of length $L_0$ which is orthogonal to
$\ell$ at the midpoint of $J$ and let the midpoint of $J$ be equal
to $I\cap \ell$. Then, any leaf of $|\lambda|$ with non-trivial
intersection with $I$ also intersects $J$.

One can check that any leaf of $|\lambda|$ ($\subset \mathcal{G}$)
which intersects $J$ is contained in a box $Q'$ with center $\ell$
satisfying $L(Q')=2\log \cosh(L_0/2)$. To see this, we identify
$\mathbb{D}$ with the upper half-plane $\mathbb{H}$ and normalize
$J$ and $\ell$ such that $J=[1,e^{L_0}]i$ and
$\ell=\{|z|=e^{L_0/2}\}\cap \mathbb{H}$. Any complete geodesic which
is disjoint from $\ell$ and which intersects $J$ is in the box
$Q'=[e^{3L_0/2},-e^{L_0/2}]\times [e^{L_0/2},e^{3L_0/2}]$. This
means that $\lambda(I)\le \lambda(J)\le \lambda(Q')$ and hence we
conclude
$$
\|\lambda\|_{Th}\le C_0\sup_Q\lambda(Q)
$$
with universal constant $C_0>0$,
where the supremum runs over all boxes $Q$ with $L(Q)=\log 2$.

To show the converse, let $Q=[a,b]\times [c,d]$ be a box  in
$\mathcal{G}$. The measure $\lambda (Q)$ is obtained as follows.
Suppose for the simplicity that $a$, $b$, $c$ and $d$ are lying on
$S^1$ in this order. Let $\ell_1=\geodesic{a}{d}$ and
$\ell_2=\geodesic{b}{c}$ and $I$ the geodesic segment which
intersects orthogonally to $\ell_1$ and $\ell_2$ at endpoints. Then,
any complete geodesic in $Q$ intersects $I$. Since the length of $I$
is $\log 2<1$, there is a geodesic arc $I'$ of unit length which
contains $I$ and hence we obtain
$$
\lambda(Q)\le \lambda(I')\le \|\lambda\|_{Th},
$$
for all boxes $Q$ with $L(Q)=\log 2$ which implies the desired
inequality.
\end{proof}

\subsection{H\"older distributions defined from measures}
Any $\lambda\in \mathcal{ML}_b(\mathbb{D})$ induces a H\"older
distribution by the formula
$$
\hol_0\ni \varphi\mapsto \int_\mathcal{G}\varphi d\lambda.
$$
Indeed,
by definition
and Lemma \ref{lem:Thurston_norm_and_Holder_distribution},
we have
$$
\|\lambda\|_\nu=
\sup_{(\varphi,Q)\in {\rm test}(\nu)}
\left|\int_Q\varphi d\lambda\right|
\le \sup_Q\lambda (Q)
\le \|\lambda\|_{Th}
$$
for all $0<\nu\le 1$, where in the third term, $Q$ runs over all
boxes $Q$ with $L(Q)=\log 2$. Thus the above formula gives a natural
inclusion of $\mathcal{ML}_b(\mathbb{D})$ into $\mathcal{H}$.

The following lemma extends the above equivalence of norms to any
locally finite Borel measure on $\mathcal{G}$.

\begin{lemma} \label{lem:Holder_top_and_Measure_top}
Let $\lambda$ be a locally finite Borel measure on $\mathcal{G}$.
Then the induced linear functional
\begin{equation} \label{eq:Holder-lambda}
\lambda:\hol_0\ni \varphi\mapsto \int_\mathcal{G}\varphi d\lambda
\end{equation}
is a H\"older distribution if and only if $\sup_Q\lambda
(Q)<\infty$, where the supremum is over all boxes $Q$ with
$L(Q)=\log 2$. In this case, there is a universal constant $C_1>0$
such that
$$
\|\lambda\|_\nu\le \sup_Q\lambda (Q)\le
 C_1\|\lambda\|_\nu,
$$
for all $\nu$ with $0<\nu\le 1$, where $L(Q)=\log 2$, and
$\|\lambda\|_\nu$ is the $\nu$-norm of the H\"older distribution
\eqref{eq:Holder-lambda}.
\end{lemma}

\begin{proof}
From \eqref{eq:special_test_function},
we obtain
$$
\lambda (Q^*_0)\le
\int_{Q^*}\varphi_0d\lambda\le
((\pi/2)^{1-\nu}\|\varphi_0\|_1)\|\lambda\|_\nu
\le C'_1\|\lambda\|_\nu,
$$
where $C'_1$ is a universal constant.
Since $Q^*$ is covered by finitely many boxes
which are the images of $Q^*_0$ under M\"obius transformations,
by applying the argument above to $(\gamma_Q)^*\lambda$
and $\varphi_0\circ \gamma_Q^{-1}$
instead of $\lambda$ and $\varphi_0$,
we conclude that
$$
\lambda (Q)\le C_1\|\lambda\|_\nu.
$$
for all $Q$ with $L(Q)=\log 2$,
where $C_1$ is a universal constant.
The left-hand side follows from the standard argument.
Indeed, since $\|\varphi \circ \gamma_Q\|_\nu\le 1$,
$\sup_Q|\varphi|\le 1$ and hence
for any $\epsilon>0$,
we can take $(\varphi,Q)\in {\rm test}(\nu)$
such that
$$
\|\lambda\|_\nu
\le \left|
\int_Q \varphi d\lambda
\right|+\epsilon
\le \lambda(Q)+\epsilon
\le \sup_Q \lambda(Q)+\epsilon,
$$
which implies what we wanted.
\end{proof}

\section{Earthquakes and Earthquake measures}
\subsection{Earthquakes}
Let $\mathcal{L}$ be a geodesic lamination in $\mathbb{D}$. An
\emph{earthquake} $E$ with the support $\mathcal{L}$ is a surjective
map $E:\mathbb{D}\to \mathbb{D}$ such that $E$ is a hyperbolic
isometry when restricted to any stratum of $\mathcal{L}$ and, for
any two strata $A$ and $B$, the \emph{comparison isometry}
$$
\rm{cmp}(A,B)= (E\mid_A)^{-1}\circ E\mid_B
$$
is a hyperbolic translation whose axis weakly separates $A$ and $B$,
and which translates $B$ to the left as seen from $A$. An earthquake
$E$ of $\mathbb{D}$ continuously extends to a homeomorphism of the
boundary $S^1$ (see \cite{Thurston}). We denote by $E\mid_{S^1}$ the
extension.

Given an earthquake $E$ with support $\mathcal{L}$, there is an
associated positive transverse measure $\lambda$ to $\mathcal{L}$ as
follows. Let $I$ be a closed geodesic arc transversely intersecting
$\mathcal{L}$ with arbitrary orientation. For given $n$, choose a
closed geodesic arc $I_n$ which contains $I$ in its interior such
that $I_{n+1}\subsetneq I_n$ and $\cap_nI_n=I$. Furthermore, choose
strata $\mathcal{A}_n=\{A_0,A_1,\cdots ,A_{k(n)},A_{k(n)+1}\}$ of
the support of $E$ such that $A_0$ contains the left end point of
$I_n$, $A_1$ contains the left endpoint of $I$, $A_{k(n)}$ contains
the right endpoint of $I$, $A_{k(n)+1}$ contains the right endpoint
of $I_n$, $A_i$'s intersect $I$ in the given order and the maximum
of the distances between the consecutive intersections of
$\mathcal{A}_n$ with $I_n$ goes to zero as $n\to\infty$. The
summation of the translation lengths of the comparison isometries
$\rm{cmp}(A_i,A_{i+1})=(E\mid_{A_i})^{-1}\circ E\mid_{A_{i+1}}$ for
$i=0,1,\cdots ,k(n)+1$ is the approximate measure of $I$. If $n\to
\infty$ and $\mathcal{A}_n$ are chosen such that
$(\cup_{i=1}^{k(n)}A_i)\cap I$ is dense in $I$ for all $n$, the
limit of approximate measure is a well-defined positive finite Borel
measure (\cite{Thurston} and \cite{GHL}). (Note that if
$E:\mathbb{D}\to\mathbb{D}$ is continuous at the endpoints of $I$
then we can replace $I_n$ with $I$ for each $n$ in the above
construction.) This transverse measure defines a measured lamination
$\lambda$ with support $\mathcal{L}$. We call the measured
lamination $\lambda$ the \emph{earthquake measure} for $E$. We
denote by $E^\lambda$ a earthquake map with earthquake measure
$\lambda$. An earthquake map is (essentially) uniquely determined by
its earthquake measure. The ambiguity is up to post-composition of
the earthquake map by a M\"obius map and on each leaf where the
earthquake has a discontinuity there is a range of possibilities
(but the extension to $S^1$ gives the same map regardless of the
choices in this range.) The set of strata where an earthquake map
has a discontinuity consists of at most countable family of leaves
of $\mathcal{L}$.

In \cite{Thurston}, Thurston showed that for any orientation
preserving homeomorphism $h$ on $\partial \mathbb{D}$, there is a
unique earthquake map $E^{\lambda}$ such that $h=E^\lambda|_{S^1}$.
Thurston's theorem induces an injective map from the space of right
cosets of ${\rm M\ddot{o}b}(\mathbb{D})$ in the group of orientation
preserving homeomorphisms into the space of measured laminations in
$\mathbb{D}$ by the formula ${\rm M\ddot{o}b}(\mathbb{D})\circ
h\mapsto \lambda$ where $h=E^{\lambda}|_{S^1}$.

For an orientation preserving homeomorphism $h:S^1\to S^1$ and the
earthquake map $E^{\lambda}|_{S^1}=h$, we have that $h\circ\gamma
=E^{\gamma^{*}(\lambda )}|_{S^1}$ for any $\gamma\in {\rm
M\ddot{o}b}(\mathbb{D})$.

\subsection{Convergence of earthquakes}
\label{subsec:remark1} Notice from the definition that for any
$\gamma\in {\rm M\ddot{o}b}(\mathbb{D})$, the earthquake measure of
$\gamma\circ E$ coincide with that of $E$. Hence, $E^\lambda$ is
determined \emph{up to} postcomposition of  M\"obius
transformations. Because of this ambiguity, we should give a remark
on the symbol $E^\lambda$. Namely, when $E^\lambda$ is treated as a
map, this $E^\lambda$ is always chosen suitably for the content. For
instance, we have used the equation ``$h=E^\lambda$" with a
homeomorphism $h$ on $S^1$.

This equation means that we can choose an earthquake map
with earthquake measure $\lambda$ which coincides with $h$
on $S^1$.
When we say that ``$E^{\lambda_n}\to E^\lambda$ as $n\to \infty$",
a sequence consisting of choices of the earthquake maps for $\lambda_n$ ($n\in \mathbb{N}$)
converges to one of those for $\lambda$.

\section{The universal Teichm\"uller space and the Earthquake measure map}
\subsection{Quasisymmetic maps}
An orientation preserving homeomorphism $h$ is said to be a
\emph{quasisymmetric} if there is a constant $M\geq 1$ such that
\begin{equation} \label{eq:quasisymetric}
\frac{1}{M}\le \frac{|h(J_1)|}{|h(J_2)|}\le M
\end{equation}
for all adjacent intervals $J_1, J_2\subset S^1$ with $|J_1|=|J_2|$,
where $|J_i|$ is the arc length with respect to the angle measure on
$S^1=\partial \mathbb{D}$. Let $\mathcal{QS}$ be the set of all
quasisymmetic maps on $S^1$. The \emph{universal Teichm\"uller
space} $T(\mathbb{D})$ is the quotient space
$$
T(\mathbb{D})=
{\rm M\ddot{o}b}(\mathbb{D})\backslash \mathcal{QS}
$$
where the group ${\rm M\ddot{o}b}(\mathbb{D})$ of M\"obius
transformations acts on $\mathcal{QS}$ via post-compositions. For
any $h\in\mathcal{QS}$, we denote by $[h]$ its class in
$T(\mathbb{D})$. The universal Teichm\"uller space $T(\mathbb{D})$
admits a natural (metric) topology inherited from the maximal
dilatations. Namely, two quasisymetric maps $h_1$ and $h_2$ are
close if there exists a quasiconformal extension of $h_2\circ
h_1^{-1}$ whose maximal dilatation is near one. This topology on
$T(\mathbb{D})$ is the same one inherited from quasisymetric
constants. See \cite{DE} or \cite{GL}.

\subsection{The earthquake measure map}
In this subsection, we define the earthquake measure map. We first
recall the following theorem, which is proved by Gardiner-Hu-Lakic
\cite{GHL} and in \cite{Saric2}.

\begin{theorem}[Gardiner-Hu-Lakic, \v{S}ari\'c]
\label{thm:earthquake_measure_map} Let $h$ be an orientation
preserving homeomorphism $h$ of $\partial \mathbb{D}=S^1$ and let
$E^{\lambda}$ be the earthquake of $\mathbb{D}$ whose continuous
extension to $S^1$ equals $h$. Then the following are equivalent.
\begin{itemize}
\item[(1)]
The earthquake measure $\lambda$ of the earthquake
$E^{\lambda}|_{S^1}=h$ is bounded.
\item[(2)]
$h$ is quasisymmetric.
\end{itemize}
\end{theorem}
The \emph{earthquake measure map}
$$
\mathcal{EM}:T(\mathbb{D})\to
\mathcal{ML}_b(\mathbb{D})
$$
is defined by $\mathcal{EM}([h])=\lambda$ where
$h=E^\lambda|_{S^1}$. As noted in \S\ref{subsec:remark1}, every
earthquake is determined by its earthquake measure up to
post-composition by M\"obius maps. Hence, together with the
uniqueness of the earthquake measures for homeomorphisms
\cite{Thurston}, Theorem \ref{thm:earthquake_measure_map} tells us
that the earthquake measure map $\mathcal{EM}$ is well-defined and
bijective.

In \cite{GHL} and \cite{Hu1}, it is proved that for a quasisymmetric
map $h$, the Thurston norm of the earthquake measure of $h$ is
comparable with the quasisymmetric constant of $h$. We will give a
brief proof of a weaker result than the comparison statement which
we need here (cf. Lemma \ref{lem:comparizon}).

\section{An example}
\label{sec:example_earthquake} In this section, we consider the
example from Introduction of non-convergence of a sequence in the
space of bounded measured laminations in the Fr\'echet topology
which converges in the weak* topology.

\subsection{Fr\'echet topology vs weak* topology}
\label{subsec:Frechet_weak-star} For the simplicity, we use the
upper half-plane model $\mathbb{H}$ for the hyperbolic plane in
place of $\mathbb{D}$. Let $\ell_n=\geodesic{1/n}{\infty}$
($n\in\mathbb{Z}\setminus \{0\})$ and
$\ell_\infty=\geodesic{0}{\infty}$ in $\mathcal{G}$.

\medskip
\noindent {\bf Example 1.} \quad Let $\lambda_n$ be the measured
lamination whose support is $\ell_n$ with $\lambda_n(\ell_n)=1$. Let
$\lambda_{\infty}$ be the measured lamination whose support is
$\ell_{\infty}$ such that $\lambda_{\infty}(\ell_{\infty})=1$. Then
$\lambda_n$ does \emph{not} converge to $\lambda_\infty$ in the
Fr\'echet topology as $n\to \infty$, while it does converge in the
weak* topology on measures on $\mathcal{G}$.

Indeed, for $n\ge 1$ and $\omega_0=(1+2\sqrt{2})^2$, we define a box
$Q_n=[-a_n,a_n]\times [\omega_0 a_n,-\omega_0 a_n]$ with
$1/(\omega_0 n)<a_n<1/n$, where $[\omega_0 a_n,-\omega_0 a_n]$ is
the interval in $\partial \mathbb{H}=\mathbb{R}\cup \{\infty\}$
which contains $\infty$ and connects $\omega_0 a_n$ and $-\omega_0
a_n$ (cf. Figure \ref{fig:atoms}).

\begin{figure}
\includegraphics[height=5cm]{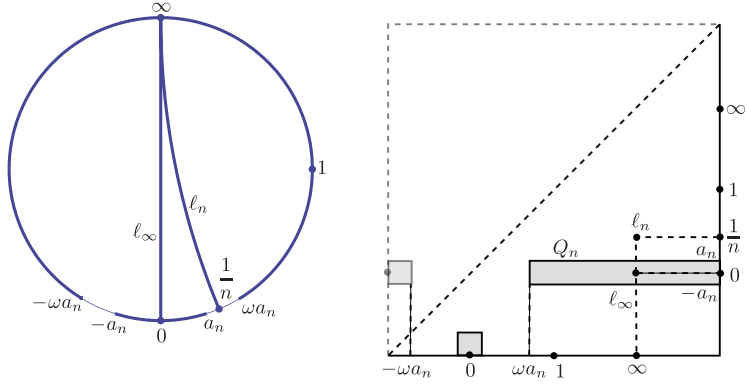}
\caption{$\ell_\infty$, $\ell_n$, and the box $Q_n$ with center
$\ell_\infty$ and $L(Q_n)=\log 2$ such that $\ell_n\notin Q_n$. The
right picture represents how $Q_n$ distributes in the space
$\mathcal{G}$.} \label{fig:atoms}
\end{figure}

Then, one can check that $L(Q_n)=\log 2$, $\lambda_\infty(Q_n)=1$
and $\lambda_n(Q_n)=0$ since $\ell_n\not\in Q_n$. We take a
Lipschitz function on $\mathcal{G}$ with support in $Q^*$ such that
$\|\varphi\|_1\le 1$ and the value at the center $\ell_{Q^*}$ of
$\varphi$ is positive. Set
$\varphi_{\nu,n}=(2/\pi)^{1-\nu}\varphi\circ (\gamma_{Q_n})^{-1}$
for $0<\nu\le 1$. From the symmetries of $Q_n$ and $Q^*$, one can
see that $\gamma_{Q_n}(\ell_\infty)=\ell_{Q_n}$ for all $n$. Thus,
by \eqref{eq:Holder_comparizon}, the pair $(\varphi_{\nu,n},Q_n)$ is
in ${\rm test}(\nu)$ and satisfies
\begin{equation} \label{eq:comparizon_2}
\|\lambda_n-\lambda_\infty\|_\nu\ge
\left|
\int_{Q_n}\varphi_{\nu,n} d(\lambda_n-\lambda_\infty)
\right|=(2/\pi)^{1-\nu}\varphi_{\nu,n}(\ell_\infty)\ge (2/\pi)\varphi (\ell_{Q^*})
\end{equation}
for all $n$ and $0<\nu\le 1$, which implies what we wanted. By the
same reason, we can see that the ``midpoint approximation``
$\frac{1}{2}(\lambda_n+\lambda_{-n})$ does not converge to
$\lambda_\infty$ in the  Fr\'echet topology either. We generalize
this example in the following proposition.

\begin{proposition} \label{prop:endpoints_single}
Let $\{\lambda_n\}_{n=1}^\infty$ be a sequence of bounded measured
laminations which converges in the Fr\'echet topology to a measured
lamination $\lambda_\infty$ whose support is a  single geodesic.
Then, for all sufficiently large $n$, each endpoint of
$|\lambda_\infty|$ is contained in the closure the set of endpoints
of leaves of $\lambda_n$.
\end{proposition}

\begin{proof}
Let $|\lambda_\infty|=\geodesic{0}{\infty}$. Suppose on the contrary
that there is a $\delta_n>0$ such that any leaf of $\lambda_n$ does
not have endpoints in an open interval $(-\delta_n,\delta_n)$. We
take a sufficiently small $a_n>0$ such that $\omega_0 a_n<\delta_n$,
where $\omega_0=(1+\sqrt{2})^2$ as before. Define $Q_n$ by
$$
Q_n=[-a_n,a_n]\times [\omega_0a_n,-\omega_0a_n]
$$
Then, the center of $Q_n$ is $\ell_\infty$, $L(Q_n)=\log 2$ and
$Q_n\cap |\lambda_n|=\emptyset$. Thus, by the same calculation as
\eqref{eq:comparizon_2}, we get
\begin{equation*} \label{eq:finite_single_1}
\|\lambda_n-\lambda_\infty\|_\nu\ge (2/\pi)\varphi(\ell_\infty)
\end{equation*}
for some Lipschitz function $\varphi$ independent of $\nu$.
This means that $\{\lambda_n\}_{n=1}^\infty$
can not converge to $\lambda_\infty$ in the Fr\'echet topology.
\end{proof}

Unfortunately, Proposition \ref{prop:endpoints_single} does not give
a characterization of bounded measured laminations in a neighborhood
of an elementary measured lamination which is illustrated by Example
1.

\subsection{Elementary Earthquakes}
We shall check the behavior of earthquakes whose supports are single
geodesics given in the above section to clarify the connection
between the Fr\'echet topology and the weak* topology on the
measured laminations and the Teichm\"uller topology on the
extensions to $S^1$ of their corresponding earthquake maps.

Let $\ell_n=\geodesic{1/n}{\infty}$ for $n\in \mathbb{N}\cup
\{\infty\}$. Then the earthquake map $E^{\lambda_n}$ for elementary
measures $\lambda_n$ with single geodesic support $\ell_n$ and mass
$1$ (normalized to fix three points $\{-1,0,\infty\}$) is
$$
E^{\lambda_n}(z)=\left\{
\begin{array}{cc}
e (z-1/n)+1/n & \mbox{(${\rm Re}(z)>1/n$)} \\
z & \mbox{(${\rm Re}(z)\le 1/n$)}
\end{array}
\right.
$$
for $z\in \mathbb{H}$, where we set $1/\infty=0$. Clearly
$h_n:=E^{\lambda_n}\mid_{\partial \mathbb{H}}$ converges to
$h_\infty=E^{\lambda_\infty}\mid_{\partial \mathbb{H}}$ pointwise.
However, $h_n$ does not converge to $h_\infty$ in the Teichm\"uller
topology. Indeed, for $n\in \mathbb{N}$ and boxes
$Q_n=[\infty,-e/n]\times [0,e/n]$, we get $L(Q_n)=\log 2$ and
$$
L(h_n\circ h_\infty^{-1}(Q_n))
=\log (e+1)-1.
$$
This means that the maximal dilatation of any quasiconformal
extension of $h_n\circ h_\infty^{-1}$ is uniformly greater than $1$.
Thus, a sequence $\{h_n\}_{n=1}^\infty$ does not converge to
$h_\infty$ in $T(\mathbb{H})$, which also follows from Theorem
\ref{thm:EMisHomeo} and Example $1$ above.

\section{The earthquake measure map is a homeomorphism}
In this section, we prove Theorem \ref{thm:EMisHomeo}. To do so, we
define a \emph{uniform-weak* topology} on
$\mathcal{ML}_b(\mathbb{D})$ (see \cite{Saric1}) and show that it is
equivalent to the restriction of the Fr\'echet topology.

\subsection{Uniform-weak* topology}
We say that a sequence $\lambda_m\in\mathcal{ML}_b(\mathbb{D})$
\emph{converges to $\lambda\in \mathcal{ML}_b(\mathbb{D})$ in the
uniform-weak* topology} if for any continuous function $f$ on
$\mathcal{G}$ with ${\rm supp}(f)\subset Q^*$,
$$
\sup_{Q}
\int_{Q^*}f d((\gamma_Q)^*(\lambda_m)-(\gamma_Q)^*(\lambda))
\to 0
$$
as $m\to \infty$, where the supremum is over all boxes $Q$ with
$L(Q)=\log 2$ and $\gamma_Q\in{\rm M\ddot{o}b}(\mathbb{D})$ is such
that $\gamma_Q(Q^{*})=Q$.

\subsection{Two lemmas}
 Let us start with the following lemma.

\begin{lemma} \label{lem:weak-conv_and_Holder}
Let $\{\lambda_m\}_{m\in\mathbb{N}}$ be a sequence of bounded
measured laminations and $\lambda$ a bounded measured lamination.
Suppose that there exists $C>0$ such that $\|\lambda_m\|_{Th}<C$ for
all $m\in\mathbb{N}$. Then, the following are equivalent.
\begin{itemize}
\item[(1)]
The sequence $\{\lambda_m\}_{m}$ converges to $\lambda\in
\mathcal{ML}_b(\mathbb{D})$ in the uniform-weak* topology.
\item[(2)]
The sequence $\{\lambda_m\}_{m}$ converges to $\lambda\in
\mathcal{ML}_b(\mathbb{D})$ in the Fr\'echet topology.
\end{itemize}
\end{lemma}

\begin{proof}
Assume that (1) holds. Seeking a contradiction, we suppose (2) does
not hold. Then, by taking a subsequence of $\{\lambda_m\}_m$ if
necessary, there are $\nu$, $\epsilon_0>0$ and a sequence
$\{(\varphi_m,Q_m)\}_{m=1}^\infty$ in ${\rm test}(\nu)$ such that
\begin{equation} \label{eq:from1to2_1}
\left|
\int_{Q^*}\varphi_m\circ \gamma_{Q_m} d\hat{\lambda}_m
\right|
=
\left|\int_{Q_m}\varphi_m d(\lambda_m-\lambda)
\right|
\ge \epsilon_0.
\end{equation}
for all $m$, where we set
$\hat{\lambda}_m=(\gamma_{Q_m}^*\lambda_m)-(\gamma_{Q_m}^*\lambda)$
for the simplicity. From the definition of a test function,
$\varphi_m\circ \gamma_{Q_m}$ satisfies $\|\varphi_m\circ
\gamma_{Q_m}\|_\nu\le 1$. Hence, by Ascoli-Arzela's theorem, the
sequence contains a convergence subsequence $\{\varphi_{m_j}\circ
\gamma_{Q_{m_j}}\}_j$ in the $C^0$-topology. We denote by
$\psi_\infty$ its limit.

Since the support of $\varphi_{m_j}\circ \gamma_{Q_{m_j}}$ is contained in $Q^*$,
so is that of $\psi_\infty$.
Notice that
\begin{equation} \label{eq:from1to2_2}
\int_{Q^*}\psi_\infty d\hat{\lambda}_{m_j}
= \int_{Q^*}\varphi_{m_j}\circ \gamma_{Q_{m_j}}
 d\hat{\lambda}_{m_j}+
 \int_{Q^*}(\psi_\infty-\varphi_{m_j}\circ \gamma_{Q_{m_j}}) d\hat{\lambda}_{m_j}.
\end{equation}
Since the Thurston norm of $\lambda_{m_j}$ is uniformly bounded, it
follows that the last term of the right-hand side of
\eqref{eq:from1to2_2} tends to zero. From \eqref{eq:from1to2_1}, we
get
$$
\sup_{Q,L(Q)=\log 2}
\left|\int_{Q^*}\psi_\infty d
((\gamma_{Q}^*\lambda_{m_j})-(\gamma_{Q}^*\lambda))
\right|
\ge
\left|
\int_{Q^*}\psi_\infty d\hat{\lambda}_{m_j}
\right|\ge \epsilon_0/2
$$
for sufficiently large $j$, which contradicts (1). Thus (1) implies
(2).

We now assume that (2) holds, and that (1) does not hold and seek a
contradiction again. Then, after taking a subsequence of
$\{\lambda_m\}_{m=1}^\infty$ if necessary, there exist
$\epsilon_0>0$ and a continuous function $f$ on $\mathcal{G}$ with
${\rm supp}(f)\subset Q^*$ such that
$$
\sup_{Q}
\left|
\int_{Q^*}f d((\gamma_Q)^*(\lambda_m)-(\gamma_Q)^*(\lambda))
\right|
\ge 2\epsilon_0
$$
for all $m$,
where the supremum is taken over all squares $Q$ with $L(Q)=\log 2$.
This implies that there is a sequence $\{Q_m\}_{m=1}^\infty$ of boxes
such that $L(Q_m)=\log 2$ and
\begin{equation} \label{eq:from2to1_1}
\left|
\int_{Q^*}f d\hat{\lambda}_m
\right|
\ge \epsilon_0
\end{equation}
for $m\ge 1$,
where we set $\hat{\lambda}_m=(\gamma_{Q_m})^*(\lambda_m)-(\gamma_{Q_m})^*(\lambda)$.

Let $\epsilon>0$. Take a $\nu$-H\"older function $\varphi_\epsilon$
with ${\rm supp}(\varphi_\epsilon)\subset Q^*$ such that the
supremum norm of $f-\varphi_\epsilon$ is less than $\epsilon$. Let
$\psi_m=(\|\varphi_\epsilon\|_\nu)^{-1}(\varphi_\epsilon\circ
\gamma_{Q_m}^{-1})$. Then, a pair $(\psi_m,Q_m)$ is in ${\rm
test}(\nu)$ and it satisfies
\begin{align}
\int_{Q_m}\psi_m\ d(\lambda_m-\lambda)
&=
\frac{1}{\|\varphi_\epsilon\|_\nu}\int_{Q^*}\varphi_\epsilon d
\hat{\lambda}_m \nonumber\\
&=
\frac{1}{\|\varphi_\epsilon\|_\nu}
\left(\int_{Q^*}f d\hat{\lambda}_m+
\int_{Q^*}(\varphi_\epsilon -f)d\hat{\lambda}_m\right).
\label{eq:from2to1_2}
\end{align}
By Lemma \ref{lem:Thurston_norm_and_Holder_distribution} and by our
assumption that Thurston norms of $\lambda_m$ are uniformly bounded,
the last term in the parentheses of \eqref{eq:from2to1_2} is less
than $C_1\epsilon$ for some $C_1>0$ independent of $m$ and
$\epsilon$ (and hence $\nu$). By \eqref{eq:from2to1_1}, we get
\begin{equation*}\label{eq:from2to1_3}
\left|
\int_{Q_m}\psi_m\ d(\lambda_m-\lambda)
\right|
\ge
\frac{1}{\|\varphi_\epsilon\|_\nu}(\epsilon_0-C_1\epsilon).
\end{equation*}
Hence,
if we take $\epsilon>0$ (and $\nu>0$)
so that $C_1\epsilon<\epsilon_0/2$,
we obtain
$$
\sup_{(\varphi,Q)\in {\rm test}(\nu)}
\left|
\int_Q\varphi d(\lambda_m-\lambda)
\right|
\ge
\left|
\int_{Q_m}\psi_m\ d(\lambda_m-\lambda)
\right|
\ge
\frac{\epsilon_0}{2\|\varphi_\epsilon\|_\nu}
$$
for all $m$, which contradicts (2), since the constant on the
right-hand side is independent of $m$. Thus (2) implies (1).
\end{proof}

We need the following lemma.

\begin{lemma} \label{lem:comparizon}
For any $C_1>0$, there is $C_2>0$ depending only of $C_1$ such that
for any bounded measured lamination $\lambda$ with
$\|\lambda\|_\nu\le C_1$ for some $\nu$ with $0<\nu\le 1$, the
quasisymmetric constant of $E^\lambda\mid_{S^1}$ is at most $C_2$.
\end{lemma}

\begin{proof}
This follows from the results in \cite{Saric4}. Indeed, Lemma
\ref{lem:Holder_top_and_Measure_top} implies that
$\|\lambda\|_{Th}<\infty$. Then the earthquake path $t\mapsto
E^{t\lambda}|_{S^1}$ is a real analytic path in the universal
Teichm\"uller space $T(\mathbb{D})$ which extends to a holomorphic
motion $\tau\mapsto E^{\tau\lambda}|_{S^1}$ of $S^1$ in
$\hat{\mathbb{C}}$. Moreover, the holomorphic motion is well-defined
for $\tau$ in a neighborhood of the real line $\mathbb{R}$ whose
shape depends only on $\|\lambda\|_{Th}$ (see \cite{Saric4}). Then
the essential supremum norm of the Beltrami coefficient of the
extension of the holomorphic motion of $S^1$ to a holomorphic motion
of $\hat{\mathbb{C}}$ for $\tau =1$ depends only on the shape of the
domain in which $\tau$ is defined. As we noted above, this in turn
only depends on $\|\lambda\|_{Th}$. Thus the quasisymmetric constant
of $E^{\lambda}|_{S^1}$ depends only on $\|\lambda\|_{Th}$ which
proves the lemma. An alternative proof would use results in
\cite{GHL} or in \cite{Hu1}.
\end{proof}

\subsection{Proof of Theorem \ref{thm:EMisHomeo}}
We first show that the earthquake measure map $\mathcal{EM}$ is
continuous. Let $[h]\in T(\mathbb{D})$ and
$\{[h_m]\}_{m=1}^\infty\subset T(\mathbb{D})$ with $[h_m]\to [h]$ as
$m\to \infty$. Let $\lambda_m=\mathcal{EM}([h_m])$ and
$\lambda=\mathcal{EM}([h])$. Then, it follows from Lemma 4.1 of
\cite{Saric1} that for any continuous function $f$ on $\mathcal{G}$
with ${\rm supp}(f)\subset Q^*$,
$$
\sup_{Q}\int_{Q^*}f d((\gamma_Q)^*(\lambda_m)-(\gamma_Q)^*(\lambda))\to 0
$$
as $m\to \infty$, where $Q$ runs over all boxes whose Liouville
measures are $\log 2$. Hence, by Lemma
\ref{lem:weak-conv_and_Holder}, we have
$$
\|\lambda_n-\lambda\|_\nu=
\sup_{(\varphi,Q)\in {\rm test}(\nu)}
\left|\int_Q\varphi d(\lambda_m-\lambda)\right|\to 0
$$
as $m\to \infty$,
for all $\nu$.
This means that $\mathcal{EM}$ is continuous.

Next, we show that the inverse $\mathcal{EM}^{-1}$ is continuous.
Suppose $\lambda_n=\mathcal{EM}([h_m])\to
\lambda=\mathcal{EM}([h])$. Assume on the contrary that
$\mathcal{EM}^{-1}$ is not continuous. Namely, there are
$\epsilon_0>0$ and a sequence $\{Q_m\}_{m=1}^\infty$ of boxes with
$L(Q_m)=\log 2$ such that
\begin{equation} \label{eq:continuity_1}
|L(h_m(Q_m))-L(h(Q_m))|\ge \epsilon_0
\end{equation}
for all $m$, where $h$ and $h_m$ are normalized to fix $1$, $i$ and
$-1$. Take M\"obius transformations $\beta_m$ and $\beta^*_m$ such
that $g_m=\beta_m\circ h_m\circ \gamma_{Q_m}$ and
$g^*_m=\beta^*_m\circ h\circ \gamma_{Q_m}$ fix $1$, $i$ and $-1$. By
\eqref{eq:continuity_1}, we have
\begin{equation} \label{eq:continuity_2}
|L(g_m(Q^*))-L(g^*_m(Q^*))|\ge \epsilon_0
\end{equation}
for all $m$. Since $\lambda_n\to \lambda$ in the  Fr\'echet
topology, it follows that $\|\lambda_n\|_{\nu}$ is uniformly
bounded. Lemma \ref{lem:comparizon} implies that the constants of
quasisymmetry of $g_m$ and $g_m^{*}$ are uniformly bounded. The
compactness of normalized quasisymmetric mappings with uniformly
bounded quasisymmetric constants imply that $g_m$ and $g^*_m$ have
two subsequences which are index by the same set that converge to
quasisymmetric mappings $g$ and $g^*$, respectively. For simplicity
of notation, we rename the subsequences to be $g_m$ and $g_m^{*}$.
By \eqref{eq:continuity_2}, $g$ does not coincide with $g^*$.

We claim

\medskip
\noindent {\bf Claim.} The limits, in the weak* topology, of a pair
of converging subsequences
$\{(\gamma_{Q_{m_j}})^{*}\lambda_{m_j}\}_{j=1}^\infty$ and
$\{(\gamma_{Q_{m_j}})^{*}\lambda\}_{j=1}^\infty$ of
$\{(\gamma_{Q_m})^{*}\lambda_m\}_{m=1}^\infty$ and
$\{(\gamma_{Q_m})^{*}\lambda\}_{m=1}^\infty$ is the same bounded
measured lamination $\lambda'$.

\begin{proof}[Proof of the Claim.]
From the compactness of probability measures under the weak*
topology, one sees that two sequences
$\{(\gamma_{Q_m})^{*}\lambda_m\}_{m=1}^\infty$ and
$\{(\gamma_{Q_m})^{*}\lambda\}_{m=1}^\infty$ contain a pair
$\{(\gamma_{Q_{m_j}})^{*}\lambda_{m_j}\}_{j=1}^\infty$ and
$\{(\gamma_{Q_{m_j}})^{*}\lambda\}_{j=1}^\infty$ of converging
subsequences in the weak* topology. Since $\lambda_m$ converges to
$\lambda$ in the Fr\'echet topology, by Lemma
\ref{lem:weak-conv_and_Holder},
$\{(\gamma_{Q_m})^{*}\lambda_m-(\gamma_{Q_m})^{*}\lambda\}_{m=1}^\infty$
converges to zero measure in the weak* sense. Hence the weak* limits
of the pair of converging subsequences
$\{(\gamma_{Q_{m_j}})^{*}\lambda_{m_j}\}_{j=1}^\infty$ and
$\{(\gamma_{Q_{m_j}})^{*}\lambda\}_{j=1}^\infty$ are same.
\end{proof}

We continue the proof of Theorem \ref{thm:EMisHomeo}. By Lemma 3.2
of \cite{Saric2}, we can choose representatives of earthquakes
$E^{(\gamma_{Q_m})^{*}\lambda_m}$ and
$E^{(\gamma_{Q_m})^{*}\lambda}$ such that the two sequences
$\{E^{(\gamma_{Q_m})^{*}\lambda_m}|_{S^1}\}_{m=1}^\infty$ and
$\{E^{(\gamma_{Q_m})^{*}\lambda}|_{S^1}\}_{m=1}^\infty$ converge to
the same (representative of) earthquake map $E^{\lambda'}|_{S^1}$
pointwise on $S^1$ (cf. \S\ref{subsec:remark1}). Then we take
M\"obius transformations $\hat{\beta}_m$ and $\hat{\beta}^*_m$ such
that $\hat{\beta}_m\circ E^{(\gamma_{Q_m})^{*}\lambda_n}$ and
$\hat{\beta}^*_m\circ E^{(\gamma_{Q_m})^{*}\lambda}$ fix $1$, $i$
and $-1$. Since the limits of two sequences
$\{E^{(\gamma_{Q_m})^{*}\lambda_n}\}_{m=1}^\infty$ and
$\{E^{(\gamma_{Q_m})^{*}\lambda_n}\}_{m=1}^\infty$ are same,
$\hat{\beta}_m$ and $\hat{\beta}^*_m$ converge the same M\"obius
transformation. Hence, the limits of $\hat{\beta}_m\circ
E^{(\gamma_{Q_m})^{*}\lambda_n}$ and $\hat{\beta}^*_m\circ
E^{(\gamma_{Q_m})^{*}\lambda}$ also agree.

On the other hand, from the definition of earthquakes we have that
\begin{align*}
\mathcal{EM}([\hat{\beta}_m\circ
E^{(\gamma_{Q_m})^{*}\lambda_m}|_{S^1}])
&=\mathcal{EM}([E^{(\gamma_{Q_m})^{*}\lambda_m}|_{S^1}])
=(\gamma_{Q_m})^{*}\lambda_m \\
&=\mathcal{EM}([h_m\circ \gamma_{Q_m}]) =\mathcal{EM}([g_m])
\end{align*}
and
\begin{align*}
\mathcal{EM}([\hat{\beta}_m\circ
E^{(\gamma_{Q_m})^{*}\lambda}|_{S^1}])
&=\mathcal{EM}([E^{(\gamma_{Q_m})^{*}\lambda}|_{S^1}])
=(\gamma_{Q_m})^{*}\lambda \\
&=\mathcal{EM}([h\circ \gamma_{Q_m}])
=\mathcal{EM}([g^*_m]).
\end{align*}
Since the earthquake measure map is bijective and all maps
$\hat{\beta}_m\circ E^{(\gamma_{Q_m})^{*}\lambda_m}$,
$\hat{\beta}_m\circ E^{(\gamma_{Q_m})^{*}\lambda}$, $g_m$, and
$g^*_m$ fix $1$, $i$ and $-1$, we conclude $\hat{\beta}_m\circ
E^{(\gamma_{Q_m})^{*}\lambda_m}|_{S^1}=g_m$ and $\hat{\beta}_m\circ
E^{(\gamma_{Q_m})^{*}\lambda}|_{S^1}=g^*_m$. However, this
contradicts that the limits $g$ and $g^*$ of $\{g_m\}_{m=1}^\infty$
and $\{g^*_m\}_{m=1}^\infty$ are distinct. The contradiction proves
Theorem 1.

\section{Approximations by discrete laminations\textbf{}}

The purpose of this section is to propose a candidate for a class of
\emph{nice} measured laminations in order to better understand the
universal Teichm\"uller space using earthquake maps. Indeed, we will
show that \emph{discrete measured laminations} are dense in
$\mathcal{ML}_b(\mathbb{D})$ with respect to the Fr\'echet topology.

\subsection{Discrete laminations}
A geodesic lamination $\mathcal{L}$ is said to be \emph{discrete} if
any compact set $K\subset\mathbb{D}$ intersects only finitely many
leaves of $\mathcal{L}$. Equivalently, $\mathcal{L}$ is a discrete
geodesic lamination if it is discrete subset of $\mathcal{G}$. A
measured lamination $\lambda$ is, by definition, \emph{discrete} if
its support $|\lambda|$ is a discrete subset of $\mathcal{G}$. To
show the density of discrete measured laminations in
$\mathcal{ML}_b(\mathbb{D})$, we give some notations needed in the
proof of the density theorem.

\subsubsection*{Extreme geodesics and peaks}
We recall that a box of geodesics is the product set $I\times J\in
\mathcal{G}$ where $I$ and $J$ are disjoint closed intervals of
$\partial \mathbb{D}=S^1$. In this proof, we generalize the notion
of boxes such that either $I$ or $J$ is allowed to be a point, open
or half-open interval. For a generalized box $Q=I\times J$, we
define the \emph{extreme geodesics} $\{\ell^1_Q ,\ell^2_Q\}$ for $Q$
as follows. Suppose that both $I$ and $J$ are non-degenerate
intervals. Let ${\rm Int}(I)=(a,b)$ and ${\rm Int}(J)=(c,d)$. Then,
we set $\ell^1_Q=\geodesic{a}{d}$ and $\ell^2_Q=\geodesic{b}{c}$.
When exactly one of the intervals is degenerate, say when $I=\{a\}$
and ${\rm Int}(J)=(c,d)$, we set $\ell^1_Q=\geodesic{a}{d}$ and
$\ell^2_Q=\geodesic{a}{c}$. When $I$ and $J$ are both degenerate,
$\ell^1_Q$ and $\ell^2_Q$ are defined to be the geodesic connecting
$I$ and $J$. See Figure \ref{fig:GeodesicBox}.
\begin{figure}
\includegraphics[height=5cm]{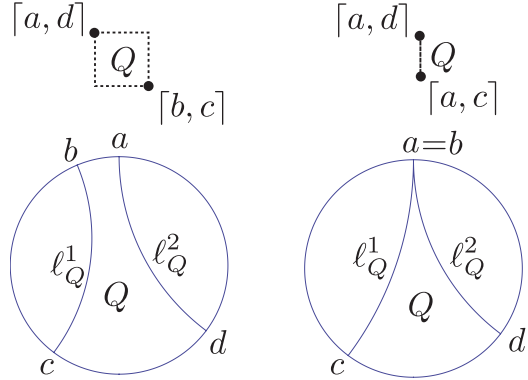}
\caption{Generalized boxes in $\mathcal{G}$ and their extreme geodesics.
}
\label{fig:GeodesicBox}
\end{figure}

Let $Q=I\times J$ be a generalized box in $\mathcal{G}$ and
$\mathcal{L}$ a geodesic lamination. Let
$\bar{Q}=\bar{I}\times\bar{J}$ be the closure of $Q$, where
$\bar{I},\bar{J}$ are closures of $I,J$. A leaf $g$ of $\mathcal{L}$
is said to be \emph{peak with respect to $Q$} if $g\in \bar{Q}$ and
one of the two components of $\mathbb{D}\setminus g$ does not
contain leaves of $\mathcal{L}\cap Q$. By definition, when
$\mathcal{L}\cap \bar{Q}$ contains at least two leaves, there is
exactly two peak geodesics of $\mathcal{L}$ with respect to $Q$. In
addition, if an extreme geodesic of $Q$ is a leaf of $\mathcal{L}$,
it is also a peak geodesic of $\mathcal{L}$ with respect to $Q$.

\subsection{Density of discrete laminations}
We are ready to prove the density of discrete laminations.

\begin{theorem}[Discrete laminations are dense] \label{thm:approximation}
The set of discrete bounded measured laminations is dense in
$\mathcal{ML}_b(\mathbb{D})$ in the Fr\'echet topology.
\end{theorem}

\begin{proof}
Fix $\lambda\in \mathcal{ML}_b(\mathbb{D})$. Let $\lambda^0$ and
$\lambda^1$ be the discrete and continuous parts of $\lambda$,
respectively. By definition, $\lambda^0$ is the sum of Dirac
measures (atoms). We identify Dirac measures appearing as terms of
$\lambda^0$ with their supports (each of them is a positive number
assigned to a point in $\mathcal{G}$).

We now fix $n$ and partition $\mathcal{G}$ into a locally finite,
countable family of
 boxes $\{B'_s\}_{s=1}^\infty$ with mutually disjoint interiors such
that $L(B'_s)\le \log 2$. We enumerate the terms of $\lambda^0$:
$$
\lambda^0=\sum_{s=1}^\infty\sum_m\mu_m^s
$$
such that ${\rm supp}(\mu_m^s)\subset B'_s$. If an atom belongs to
the boundary side of a box, then it is shared by at least two boxes
and at most four boxes. We fix one of the possible boxes to which
the atom belongs and write it in the above sum only once. It is
possible that $\{\mu_m^s\}_m$ consists of infinitely many Dirac
measures, for any $s$. For each $s$, we take $m_{s,n}$ such that
\begin{equation} \label{eq:atoms}
\sum_{s=1}^{\infty}\sum_{m\ge m_{s,n}}\mu_m^k(B'_s)<1/n.
\end{equation}
Notice from the definition that
$$
\lambda^0_n:=\sum_{s=1}^\infty
\sum_{m\le m_{s,n}}\mu_m^s
$$
is a discrete sub-measured lamination of $\lambda$. We define a
measured lamination $\lambda^1_n$ by
$$
\lambda^1_n:=\lambda-\lambda^0_n
=\lambda^1+\sum_{s=1}^\infty\sum_{m> m_{s,n}}\mu_m^k
$$
We claim the following.

\medskip
\noindent {\bf Claim 1.} For any $n$, there is a locally finite
collection $\{B_k^n\}_{k=1}^\infty$ of countably many, mutually
disjoint generalized boxes with the following properties.
\begin{itemize}
\item[(1)]
$\{B_k^n\}_{k=1}^\infty$ covers $|\lambda^1_n|$.
\item[(2)]
$\lambda^1_n(B_k^n)<1/n$ and $L(B_k^n)\le \log 2$ for all $k$, and
\item[(3)]
extreme geodesics of $B_k^n$ are leaves of $|\lambda^1_n|$.
\end{itemize}

\begin{proof}[Proof of Claim 1]
By the definition of $\lambda^1_n$, we can divide each $B'_s$ into a
finite collection of non-degenerate closed boxes such that its
$\lambda^1_n$-measure is less than $1/n$ and interiors of distinct
boxes are disjoint. We define a sub-collection
$\{{B'}_k^n\}_{k=1}^\infty$ to consist of all the above boxes
(running all $s$) which intersect the support $|\lambda^1_n|$ of
$\lambda^1_n$.

We now fix one box ${B'}_k^n$ and modify it appropriately to get the
collection of generalized boxes as in the claim.

\medskip
\paragraph{{\bf Case 1.1 :
${B'}_k^n\cap |\lambda^1_n|$ consists of one point.}}\ When
${B'}_k^n\cap |\lambda^1_n|$ is not an atom, then it has to belong
to a boundary side ${B'}_k^n$. We drop ${B'}_k^n$ from the family of
boxes. Suppose ${B'}_k^n\cap |\lambda^1_n|$ is an atom
$\lambda'_{k,n}$ of $\lambda$, we again drop ${B'}_k^n$ from the
collection of boxes and add $\lambda'_{k,n}$ to $\lambda^0_n$. Since
$\{{B'}_k^n\}_{k=1}^\infty$ is locally finite, even if we continue
this procedure infinitely (but countably) many times, $\lambda^0_n$
is still a locally finite sublamination of $\lambda$ (cf. (1) in
Figure \ref{fig:BoxesQ}).
\begin{figure}
\includegraphics[height=4cm]{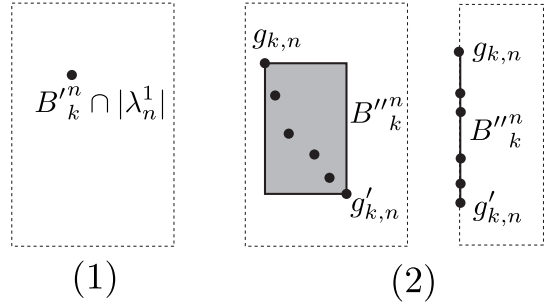}
\caption{Boxes bounded by broken lines represents ${B'}_k^n$.
}
\label{fig:BoxesQ}
\end{figure}
%

\medskip
\paragraph{{\bf Case 1.2 :
${B'}_k^n\cap |\lambda^1_n|$ contains at least two points.}}\ Let
$g_{k,n}$ and $g'_{k,n}$ be peak geodesics of $|\lambda^1_n|$ with
respect to ${B'}_k^n$. We replace the box ${B'}_k^n$ by a box
${B''}_k\subset {B'}_k^n$ whose extreme geodesics are $g_{k,n}$ and
$g'_{k,n}$ (cf. (2) in Figure \ref{fig:BoxesQ}). If it happens that
$g_{k,n}$ and $g'_{k,n}$ share the same endpoint, then ${B''}_k^n$
is a generalized box in our sense (cf. the right figure of (2) in
Figure \ref{fig:BoxesQ}).

\medskip
From the definition, the family $\{{B''}_k^n\}_{k=1}^\infty$ of the
resulting boxes is locally finite and satisfies the properties (1),
(2) and (3) in the claim.

It is possible that some of the obtained closed boxes intersect
along their boundaries. In this case, we divide the closed box into
an open box which is the interior and into boundary sides which are
generalized boxes. Each of the boundary sides is divided further
into finitely many generalized boxes such that the new family of
generalized boxes is pairwise mutually disjoint. Thus, after
renumbering with respect to $k$ if necessary, we finally obtain the
family of generalized boxes $\{B_k^n\}_{k=1}^\infty$ as we claimed.
\end{proof}

Let us continue the proof of the density theorem. Fix $n\in
\mathbb{N}$. Let $\{B_k^n\}_{k=1}^\infty$ be the family of boxes
from Claim 1. We fix $g_k^n\in B_k^n\cap |\lambda|$ arbitrary, and
define
\begin{align*}
\lambda^2_n &:=\sum_{k=1}^\infty \lambda_n^1(B_k^n)\cdot
\delta_{g_k^n} \quad
\mbox{and} \\
\lambda_n &:=\lambda^0_n+\lambda^2_n,
\end{align*}
where $\delta_{g_k^n}$ is the dirac measure on $\mathcal{G}$ with
support $g_k^n$. Since $\{B_k^n\}_{k=1}^\infty$ is locally finite,
so is $\lambda_n$. Furthermore, $\lambda_n$ is a measured geodesic
lamination, because leaves of $\lambda_n$ are leaves of $\lambda$.

We will prove that $\lambda_n$ converges to $\lambda$ in the
Fr\'echet topology, which implies that discrete bounded measured
laminations are dense in $\mathcal{ML}_b(\mathbb{D})$. We need the
following claim to show the convergence.

\medskip
\noindent {\bf Claim 2.} The following holds.
\begin{itemize}
\item[(1)]
For any box $Q$ in $\mathcal{G}$, there are at most two boxes from
the family $\{B_k^n\}_{k=1}^{\infty}$ such that $B_k^n\cap
Q\neq\emptyset$ but $B_k^n\not\subset Q$.
\item[(2)]
The sequence $\{\lambda_n\}_{n=1}^\infty$ has uniformly bounded
Thurston norms. In particular, $\lambda_n\in
\mathcal{ML}_b(\mathbb{D})$.
\end{itemize}

\begin{proof}[Proof of Claim 2]
(1)\quad Let $B_k^n$ be a box satisfying $g_k^n\in Q$ but
$B_k^n\not\subset Q$. Let $Q=[a,b]\times [c,d]$ and
$B_k^n=[x,y]\times [z,w]$. Without loss of generality, we may assume
that $b$ is in the interior of $[x,y]$. Then, there is no box
$B^n_{k'}=I'\times J'$ such that $B^n_{k'}\cap Q\neq\emptyset$ and
$I'\cap [c,z]\neq\emptyset$ or $J'\cap [c,z]\neq\emptyset$. This
follows because the extreme geodesics of $B^n_{k'}$ are contained in
a component of $\mathbb{D}\setminus \geodesic{y}{z}$ whose closure
contains $c$, and hence, no geodesic in $B^n_{k'}$ can connect
$[a,b]$ and $[c,d]$. (Figure \ref{fig:geodesic_Q_ends}).
\begin{figure}
\includegraphics[height=4cm]{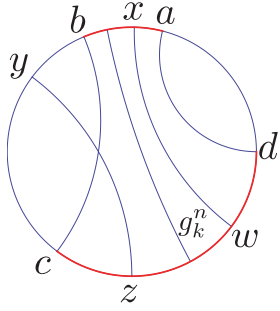}
\caption{(1) in Claim 2 of the proof of Theorem \ref{thm:approximation}.}
\label{fig:geodesic_Q_ends}
\end{figure}
If there is another box $B^n_{k_1}=[x_1,y_1]\times [z_1,w_1]$ such
that $g^n_{k_1}\in Q$ and $B^n_{k_1}\not\subset Q$, then either
$a\in [x_1,y_1]$ or $d\in [x_1,y_1]$ or $a\in [z_1,w_1]$ or $d\in
[z_1,w_1]$. The above reasoning implies that there could be no more
boxes with the above property. Thus, there are at most two boxes
with the property that $B_k^n\cap Q\neq\emptyset$ but
$B_k^n\not\subset Q$.

\medskip
\noindent (2)\quad Fix $\nu$ with $0<\nu\le 1$. Let $(\varphi,Q)\in
{\rm test}(\nu)$. From (1) of the claim, we get
\begin{align*} \label{eq:bounded_norms_2}
\int_Q\varphi d\lambda_n &\leq \lambda^0_n(Q)+ \sum_{g_k^n\in
Q}\varphi (g_k^n)\lambda^1_n(B_k^n) \le
\lambda(Q)+\sum_{B_k^n\cap Q\ne \emptyset}\lambda^1_n(B_k^n) \\
&\le
\lambda(Q)+(\lambda(Q)+(1/n)\times 2)
\le 2\sup_{Q}\lambda(Q)+2,
\end{align*}
because $\varphi(g_k^n)\le \|\varphi\|_\nu\le 1$ and $\lambda^0_n$
is a sub-measured lamination of $\lambda$, where $Q$ in the last
term runs over all boxes with $L(Q)=\log 2$. By Lemma
\ref{lem:Holder_top_and_Measure_top}, we deduce that the sequence
$\{\lambda_n\}_{n=1}^\infty$ has uniformly bounded Thurston norms.
\end{proof}

Let us continue with the proof that $\lambda_n$ converges to
$\lambda$ in the Fr\'echet topology. Let $Q$ be a square with
$L(Q)=\log 2$ and $f$ be a continuous function on $\mathcal{G}$
whose support is in $Q^*$. Let $\epsilon>0$. We take $\delta>0$ such
that $|f(\ell)-f(\ell')|<\epsilon$ when $d(\ell,\ell')\le \delta$,
where $d$ is the fixed metric on $\mathcal{G}$ induced by the angle
metric on $S^1$ with respect to $0\in\mathbb{D}$ (cf.
\S\ref{subsec:space_of_geodesics}).

Take $B_k^n$ with $Q\cap B_k^n\ne \emptyset$. Let
$\gamma_Q^{-1}(B_k^n)=I\times J$. Suppose that $I\cap [-i,1]$ and
$J\cap [i,1]$ are non-empty. We set
\begin{equation*} \label{eq:lambda_Q_n}
\hat{\lambda}_{Q,n}:= (\gamma_Q)^*(\lambda_n)-d(\gamma_Q)^*(\lambda)
= (\gamma_Q)^*(\lambda_n^2)-(\gamma_Q)^*(\lambda^1_n)
\end{equation*}
for the simplicity. We consider the following three cases for
$B_k^n$.

\medskip
\noindent
{\bf Case 1.}\
$B_k^n\subset Q$ and the length of $I$ and $J$ are less than $\delta$.

\medskip
In this case,
we have
\begin{align*}
\left| \int_{\gamma_Q^{-1}(B_k^n)} f\, d\hat{\lambda}_{Q,n} \right|
&=
\left|
f(\gamma_Q^{-1}(g_k^n))\lambda^1_n(B_k^n)
-
\int_{\gamma_Q^{-1}(B_k^n)}
f\,
d((\gamma_Q^{-1})^*(\lambda^1_n))\right| \\
&\le
\epsilon \lambda^1_n(B_k^n).
\end{align*}
Therefore, the summation over all boxes $B_k^n$ in this case gives
\begin{equation} \label{eq:convergence_case1}
\sum_{\left\{\mbox{$B_k^n$'s in Case 1}\right\}}
\left|
\int_{\gamma_Q^{-1}(B_k^n)}
f\,
d\hat{\lambda}_n
\right|
\le \epsilon \lambda^1_n(Q)
\le \epsilon \lambda(Q).
\end{equation}

\medskip
\noindent {\bf Case 2.}\ $B_k^n\subset Q$ and, if
$\gamma_Q^{-1}(B_k^n)=I\times J$ then either $I$ or $J$ has length
at least $\delta$.

\medskip
Notice that since $Q^*$ is a fixed box,
the number of such $B_k^n$ in this case is $O(1/\delta)$.
Hence,
we have
\begin{align}
\sum_{\left\{\mbox{$B_k^n$'s in Case 2}\right\}}
\left|
\int_{\gamma_Q^{-1}(B_k^n)}
f\,
d\hat{\lambda}_{Q,n}
\right|
&=
O\left(\sum (\lambda^2_n(B_k^n)+\lambda^1_n(B_k^n))\|f\|_\infty\right) \nonumber
\\
&=
O\left(\|f\|_\infty/(n\delta)\right)
\label{eq:convergence_case2}
\end{align}

\medskip
\noindent
{\bf Case 3.}\
$B_k^n\not\subset Q$.

\medskip
Notice that
\begin{equation*} \label{eq:convergence_case3-1}
\left|
\int_{\gamma_Q^{-1}(B_k^n)}
f\,
d\hat{\lambda}_n
\right|
\le (\lambda^2_n(B_k^n)+\lambda^1_n(B_k^n))\|f\|_\infty
\le 2\|f\|_\infty/n.
\end{equation*}
By (1) of Claim 2, there are at most two such boxes. Hence, we have
\begin{equation} \label{eq:convergence_case3}
\sum_{\left\{\mbox{$B_k^n$'s in Case 3}\right\}}
\left|
\int_{\gamma_Q^{-1}(B_k^n)}
f\,
d\hat{\lambda}_{Q,n}
\right|
\le 4\|f\|_\infty/n.
\end{equation}

We can now complete the proof of the convergence. Indeed, we take
$n$ sufficiently large such that $n\delta>1/\epsilon$. Then, from
the three cases above and Lemma
\ref{lem:Holder_top_and_Measure_top}, we conclude
\begin{align*}
\sup_Q\left|
\int_{Q^*}
f\,
d\hat{\lambda}_{Q,n}
\right|
&\le
\sup_Q\left\{
\sum_{\{B_k^n\cap Q\ne\emptyset\}}
\left|
\int_{\gamma_Q^{-1}(B_k^n)}
f\,
d\hat{\lambda}_{Q,n}
\right|
\right\}\\
&\le
\sup_Q
\left\{\epsilon \lambda(Q)+
O\left(\|f\|_\infty/(n\delta)\right)
+4\|f\|_\infty/{n}\right\} \\
&=
\epsilon \left(\sup_Q \lambda(Q)\right)+O(\epsilon)
=O(\epsilon),
\end{align*}
where the supremum is taken over all $Q$ with $L(Q)=\log 2$. Since
$\hat{\lambda}_{Q,n}=(\gamma_Q)^*(\lambda_n)-(\gamma_Q)^*(\lambda)$
and $\{\lambda_n\}_{n=1}^\infty$ is uniformly bounded, from Lemma
\ref{lem:weak-conv_and_Holder}, we have that $\lambda_n$ converges
to $\lambda$ in the Fr\'echet topology.
\end{proof}

\vskip .2 cm

Theorem \ref{thm:approximation} and Theorem \ref{thm:EMisHomeo}
immediately imply Theorem \ref{thm:CountableEarthquakeTheorem}.

\section{Infinitesimal Earthquakes and Vector fields}
\label{sec:infinitesimal_earthquake} In this section, we consider
the vector fields on $\partial \mathbb{D}$ which arise by
differentiating the paths of earthquakes. The aim is to prove the
equivalence between the Fr\'echet topology on earthquake measures
and the Zygmund topology on the vector fields (cf. Theorem
\ref{thm:Frechet_and_Zygmund}) which is an analogy to Theorem
\ref{thm:EMisHomeo}.

\subsection{Vector fields}
Let $\lambda$ be a bounded measured lamination. From now on, we fix
a stratum $A$ of $\lambda$ such that $A$ is either a gap or a
geodesic which is not an atom of $\lambda$. Every leaf $\ell$ of
$\lambda$ is oriented as a part of the boundary of the component of
$\mathbb{D}\setminus \ell$ containing $A$. Let $a$ be the initial
point and $b$ the terminal point of $\ell$ for the given
orientation. Let $[a,b]$ be an oriented interval connecting
endpoints of $\ell$ (cf. Figure \ref{fig:geodesic_orientation}).
\begin{figure}
\includegraphics[height=4cm]{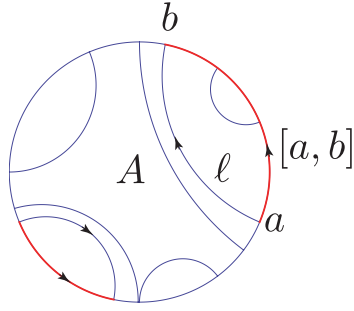}
\caption{Orientations of leaves and associated intervals.}
\label{fig:geodesic_orientation}
\end{figure}
Then,
we set
\begin{equation*}
\dot{E}^\lambda_{\ell}(z)= \left\{
\begin{array}{cl}
0 & \mbox{for $z$ outside of $[a,b]$} \\
\frac{(z-a)(z-b)}{a-b} & \mbox{for $z\in [a,b]$}.
\end{array}
\right.
\end{equation*}
When $\ell$ is not a leaf of $\lambda$, we put
$\dot{E}^\lambda_\ell(z)=0$ for all $z\in \partial \mathbb{D}$. For
any point $z\in \partial \mathbb{D}$, $\dot{E}^\lambda_{\ell}(z)$ is
a function of $\ell\in\mathcal{G}$.

We consider the integral
\begin{equation} \label{eq:integral_lamination}
\dot{E}^\lambda(z):= \int_{\mathcal{G}}\dot{E}^\lambda_\ell
(z)d\lambda(\ell)
\end{equation}
for a measured lamination $\lambda$. For a finite lamination
$\lambda=\sum_{i=1}^m\lambda_i\ell_i$, by definition, it holds
$$
\dot{E}^\lambda(z)=\sum_{i=1}^m\lambda_i\dot{E}^\lambda_{\ell_i}(z).
$$
One can show that the integral $\dot{E}^\lambda$ in
\eqref{eq:integral_lamination} is well-defined for all $\lambda\in
\mathcal{ML}_b(\mathbb{D})$ by an approximation argument (see
\cite{gardiner}). We give a more direct proof of the convergence of
the integral in Appendix (cf. \S\ref{sec:convergence_integral}).

\subsubsection{Infinitesimal earthquakes}
For $\lambda\in \mathcal{ML}_b(\mathbb{D})$ and $t>0$, we normalize
$E^{t\lambda}$ to be the identity on the stratum $A$ which we have
fixed before. Gardiner-Hu-Lakic \cite{GHL} proved that the integral
\eqref{eq:integral_lamination} gives the tangent vector fields to
the paths of earthquake deformations:
\begin{equation} \label{eq:equation_E_lambda}
\dot{E}^\lambda(z)= \left.\frac{d}{dt}E^{t\lambda}(z) \right|_{t=0}
\end{equation}
for $z\in \partial \mathbb{D}$ (cf. \cite{GHL}). Let $\zyg(\partial
\mathbb{D})$ be the Banach space of Zygmund functions on
$\partial\mathbb{D}$ modulo the subspace of quadratic polynomials
(cf. \S\ref{subsec:Frechet_Zygmund}). Gardiner \cite{gardiner} also
proved the \emph{infinitesimal earthquake theorem}, which states
that the map
\begin{equation} \label{eq:gardiner-earthquake}
\mathcal{ML}_b(\mathbb{D})\ni \lambda\mapsto \dot{E}^\lambda\in \zyg(\partial \mathbb{D})
\end{equation}
is bijective (Theorem 5.1 of \cite{gardiner}).

\subsubsection{Convergence of Vector fields}
The following proposition is well-known.

\begin{proposition}
\label{prop:convergence_of_integral} Let $\lambda\in
\mathcal{ML}_b(\mathbb{D})$ and let $\{\lambda_n\}_{n=1}^\infty$ be
a sequence in $\mathcal{ML}_b(\mathbb{D})$ with uniformly bounded
Thurston norms. If $\{\lambda_n\}_{n=1}^\infty$ converges to
$\lambda$ in the weak* topology, then $\dot{E}^{\lambda_n}$
pointwise converges to $\dot{E}^\lambda$ on $\partial \mathbb{D}$.
\end{proposition}

We shall give a proof of Proposition
\ref{prop:convergence_of_integral} in Appendix
(\S\ref{subsec:weakconv_pointwise}) for the completeness. After
that, we will give a simple proof of the formula
\eqref{eq:equation_E_lambda} using holomorphic motions and
Proposition \ref{prop:convergence_of_integral} in
\S\ref{subsec:differential_formula}.

\subsection{Fr\'echet and Zygmund}
\label{subsec:Frechet_Zygmund}
Let $V$ be a continuous function on $\partial \mathbb{D}$
safisfying $V(z)/(iz)\in \mathbb{R}$
for $z\in \partial \mathbb{D}$.
We say that $V$ is in \emph{the Zygmund class} if
there is an $M>0$ such that
\begin{equation} \label{eq:zygmund_norm}
|V(e^{i(x+t)})+V(e^{i(x-t)})-2V(e^{ix})|\le M|t|
\end{equation}
for all $0\le x<2\pi$ and $0<t<\pi$. The infimum of the constant $M$
in \eqref{eq:zygmund_norm} is called the \emph{Zygmund norm} of $V$
and we denote it by $\|V\|_{Zyg}$. Recall that $\|V\|_{Zyg}=0$ if
and only if $V$ is a quadratic polynomial. The quotient of the class
of continuous functions satisfying $V(z)/(iz)\in \mathbb{R}$ for
$z\in
\partial \mathbb{D}$ and inequality (\ref{eq:zygmund_norm})
by the subspace consisting of the quadratic polynomials becomes a
Banach space $\zyg (\partial \mathbb{D})$ with the norm $\|\cdot
\|_{Zyg}$. We call $\zyg(\partial \mathbb{D})$ the \emph{Zygmund
space}.

We define the \emph{cross-ratio norm} on $\zyg(\partial \mathbb{D})$
as follows. Let $Q=[a,b]\times [c,d]$ be a box of geodesics such
that $4$-points $a,b,c,d$ lie on $\partial \mathbb{D}$ in the
counter-clockwise. For $V\in \zyg(\partial \mathbb{D})$, we set
$$
V[Q]=\frac{V(a)-V(c)}{a-c}+\frac{V(b)-V(d)}{b-d}
-\frac{V(a)-V(d)}{a-d}-\frac{V(b)-V(c)}{b-c}.
$$
Then,
the \emph{cross-ratio norm} $\|V\|_{cr}$ of $V$ is defined by
$$
\|V\|_{cr}=\sup_{Q}|V(Q)|
$$
where $Q$ runs all boxes with $L(Q)=\log 2$. The Zygmund norm is
equivalent to the cross-ratio norm on $\zyg (\partial \mathbb{D})$
(see \cite{GL}).

\subsection{Proof of Theorem \ref{thm:Frechet_and_Zygmund}}
By Gardiner's infinitesimal earthquake theorem the map
\eqref{eq:gardiner-earthquake} is bijective. Hence it suffices to
show that the map and its inverse are both continuous.

We first check that the map \eqref{eq:gardiner-earthquake} is
continuous. Let $\lambda_n\to\lambda$ as $n\to\infty$ in the
Fr\'echet topology. Then $\|\lambda_n\|_{Th}$ is uniformly bounded.
It follows that the sequence $V_n:=\dot{E}^{\lambda_n}|_{S^1}$ has
uniformly bounded cross-ratio norms. Indeed, the cross-ratio norm
gives the infinitesimal change in the cross-ratios under the
earthquake path $t\mapsto E^{t\lambda_n}|_{\partial\mathbb{D}}$.
Assume on the contrary that $\|V_n\|_{cr}\to\infty$ as $n\to\infty$.
Then there exists a sequence $Q_n$ of boxes in $\mathcal{G}$ with
$L(Q_n)=\log 2$ such that $|V_n[Q_n]|\to\infty$ as $n\to\infty$. Let
$\gamma_{Q_n}:Q^{*}\mapsto Q_n$ be M\"obius and let
$\lambda_n':=(\gamma_n)^{*}(\lambda_n)$. Then there exists a
subsequence of $\lambda_n'$, denoted by $\lambda_n'$ for simplicity,
which converges in the weak* topology to a bounded measured
lamination $\lambda'$. Then, by Proposition
\ref{prop:convergence_of_integral}, there exist an appropriate
normalization of the earthquake vector fields such that
$\dot{E}^{\lambda_n'}|_{S^1}\to\dot{E}^{\lambda'}|_{S^1}$ pointwise
as $n\to\infty$. Since
$|V[Q_n]|=|\dot{E}^{\lambda_n'}|_{S^1}[Q^{*}]|\to\infty$ as
$n\to\infty$, this gives a contradiction. Thus the vector fields
$V_n$ have uniformly bounded cross-ratio norms.

A family of normalized Zygmund bounded maps whose cross-ratio norms
are uniformly bounded is a normal family (see \cite{GL}). If
necessary, we normalize $\dot{E}^{\lambda_n}|_{S^1}$ by adding a
quadratic polynomial, such that $\dot{E}^{\lambda_n}|_{S^1}$ is a
normal family. Assume on the contrary that
$\dot{E}^{\lambda_n}|_{S^1}\nrightarrow\dot{E}^{\lambda}|_{S^1}$ in
the cross-ratio norm topology. Then there are $C>0$ and a sequence
of quadruples $Q_n$ in $S^1$ with $L(Q_n)=\log 2$ such that
$|\dot{E}^{\lambda_n}[Q_n]-\dot{E}^{\lambda}[Q_n]|\geq C$. Let
$\gamma_{Q_n}$ be a M\"obius map such that $\gamma_{Q_n}:Q^{*}\to
Q_n$, where $Q^{*}=[-i,1]\times [i,-1]$. Then
$|\gamma_{Q_n}^{*}(\dot{E}^{\lambda_n})[Q^{*}]-
\gamma_{Q_n}^{*}(\dot{E}^{\lambda})[Q^{*}]|\geq C$ for all $n$.
Since $\|\gamma_{Q_n}^{*}(\lambda_n)\|_{Th}=\|\lambda_n\|_{Th}$ and
$\|\gamma_{Q_n}^{*}(\lambda )\|_{Th}=\|\lambda\|_{Th}$, it follows
that the Thurston norms of $\gamma_{Q_n}^{*}(\lambda_n)$ and
$\gamma_{Q_n}^{*}(\lambda )$ are uniformly bounded. Therefore, we
can extract convergent subsequences of $\gamma_n^{*}(\lambda_n)$ and
$\gamma_n^{*}(\lambda )$ in the weak* topology, which we denote by
the same letters for simplicity. The assumption on the convergence
$\lambda_n\to\lambda$ in the Fr\'echet topology implies that the
limit of $\gamma_n^{*}(\lambda_n)$ equals to the limit of
$\gamma_n^{*}(\lambda )$. On the other hand, the two sequences of
vector fields $\gamma_n^{*}(\dot{E}^{\lambda_n})$ and
$\gamma_n^{*}(\dot{E}^{\lambda})$ converge pointwise to different
limits (even different up to addition of a quadratic polynomial)
because they differ on $Q^{*}$. This implies that a single measured
lamination represents two different earthquake vector fields which
is impossible. Thus the map $\lambda\mapsto
\dot{E}^{\lambda}|_{S^1}$ is continuous.

It remains to show that the inverse map is continuous. Assume that
$\dot{E}^{\lambda_n}|_{S^1}\to\dot{E}^{\lambda}|_{S^1}$ as
$n\to\infty$ in the cross-ratio norm. We claim that there exists
$C>0$ such that $\|\lambda_n\|_{Th}<C$ for all $n$. Suppose on the
contrary that $\|\lambda_n\|_{Th}\to\infty$ as $n\to\infty$. Then
there exists a sequence $I_n$ of closed geodesic arcs whose length
is $1/n$ such that the $\lambda_n$-mass of the geodesics
intersecting $I_n$ goes to infinity as $n\to\infty$. Let $l_n$ and
$r_n$ be the leftmost and the rightmost geodesic of $|\lambda_n|$
which intersect $I_n$. It is possible that $l_n=r_n$. Let $\gamma_n$
be a M\"obius map such that the endpoints of $\gamma_n(l_n)$ are
fixed points $b<d$ in $\mathbb{R}$ and such that the endpoints of
$\gamma_n(r_n)$ converge to $b$ and $d$, respectively. Let $a<b$ and
$b<c<d$ be such that box $Q=[a,b]\times [c,d]$ satisfies $L(Q)=\log
2$. We normalize
$\dot{E}^{(\gamma_n^{-1})^{*}(\lambda_n)}|_{S^1}=(\gamma_n^{-1})^{*}(\dot{E}^{\lambda_n}|_{S^1})$
by orienting all the leaves of $|\gamma_n(\lambda_n)|$ to the left
with respect to the geodesic with endpoints $(b,d)$.

The cross-ratio norm is invariant under the push-forward by M\"obius
maps. This implies that
$\|\dot{E}^{(\gamma_n^{-1})^{*}(\lambda_n)}|_{S^1}\|_{cr}=\|\dot{E}^{\lambda_n}|_{S^1}\|_{cr}$
is bounded. Let
$V_n=\dot{E}^{(\gamma_n^{-1})^{*}(\lambda_n)}|_{S^1}$ for short. The
normalization that we imposed on $V_n$ gives that
$$
V_n[Q]=V_n(a)[\frac{1}{a-c}-\frac{1}{a-b}]+V_n(c)[\frac{-1}{a-c}+
\frac{-1}{c-d}].
$$
Both terms are non-negative. Moreover, $V_n(c)\geq
\lambda_n(I_n)\to\infty$ as $n\to\infty$, where $\lambda_n(I_n)$ is
the $\lambda_n$-mass of geodesics intersecting $I_n$. Thus
$V_n[Q]\to\infty$ as $n\to\infty$ which is a contradiction. Thus
$\|\lambda_n\|_{Th}$ is uniformly bounded.

Assume on the contrary that $\lambda_n\nrightarrow\lambda$ as
$n\to\infty$ in the Fr\'echet topology. Then, after possibly taking
a subsequence and renaming it, there exists a sequence $Q_n$ of
quadruples on $\hat{\mathbb{R}}$ such that $L(Q_n)=\log 2$ and
\begin{equation}
\label{eq: non-converg}
|\dot{E}^{\lambda_n}|_{S^1}[Q_n]-\dot{E}^{\lambda}|_{S^1}[Q_n]|\geq
c>0.
\end{equation}

Let $\gamma_n$ be M\"obius map which maps $Q=(-a,-1,1,a)$ onto
$Q_n$, where $a>1$ is chosen such that $L(Q)=\log 2$. Let
$\mu_n=(\gamma_n)^{*}(\lambda_n)$ and $\xi_n=(\gamma_n)^{*}(\lambda
)$. Since $\|\mu_n\|_{Th}$ and $\|\xi_n\|_{Th}$ are uniformly
bounded, there exist two subsequences of $\mu_n$ and $\xi_n$ with
common indexing which converge in the weak* topology. We can assume
that $\mu_n$ and $\xi_n$ converge in the weak* topology to $\mu$ and
$\xi$, respectively. By (\ref{eq: non-converg}) we get that
$|\dot{E}^{\mu}|_{S^1}[Q]-\dot{E}^{\xi}|_{S^1}[Q]|\geq c>0$ which
implies that $\mu\neq\xi$. On the other hand, since
$\dot{E}^{\lambda_n}|_{S^1}\to\dot{E}^{\lambda}|_{S^1}$ in the
Fr\'echet topology, it follows that if the push-forwards of
$\dot{E}^{\lambda_n}|_{S^1}$ and $\dot{E}^{\lambda}|_{S^1}$ by a
sequence of M\"obius maps pointwise converge then the limits have to
be equal. This is a contradiction with $\mu\neq\xi$ by the
uniqueness of the earthquake measures. Thus $\lambda_n\to\lambda$ as
$n\to\infty$ in the Fr\'echet topology which is what we needed.

\section{Appendix : The integral $\dot{E}^\lambda$}
\label{sec:convergence_integral} In this section, we consider the
integral presentation of the earthquake vector field. We prove (see
\S\ref{subsec:well-definedness_integral}) that the integral in
\eqref{eq:integral_lamination} is well-defined.

\subsection{Strata and restricted measures}
Recall that a stratum of a (measured) geodesic lamination $\lambda$
is either a leaf of $\lambda$ or the closure of a compoment of
$\mathbb{D}\setminus \lambda$. By a \emph{generalized stratum}, we
mean either a stratum of $\lambda$ or a point of $\partial
\mathbb{D}$.

Let $\lambda$ be a measured lamination. Let $A$ and $B$ be two
generalized strata of $\lambda$. We denote by $\lambda_{A,B}$ a
measured lamination whose support consists of leaves of $\lambda$
separating $A$ and $B$ in $\mathbb{D}$, and a leaf in $\partial A$
(resp. $\partial B$) facing $B$ (resp. $A$), if $A$ (resp. $B$) is a
gap. The measure is defined to be the restriction of $\lambda$ on
the above set of geodesics. Thus, $\lambda_{A,B}$ is a measured
geodesic lamination.

Alternatively, take a geodesic $I$ connecting $A$ and $B$ where
$A\cap I$ and $B\cap I$ are points. When either $A$ or $B$, say $B$,
is a point of $\partial \mathbb{D}$, we set $I$ to be a geodesic ray
from a point of $A$ terminating at $B$ such that $A\cap I$ consists
of a point. We can define $I$ in the similar way when both $A$ and
$B$ are points of $\partial \mathbb{D}$. Let $|\lambda|_I$ be leaves
of $\lambda$ intersecting $I$. Notice that the set $|\lambda|_I$ is
 independent of the choice of the geodesic $I$. Since $I$
is closed, $|\lambda|_I$ is a geodesic lamination, that is, it is a
closed subset of $\mathcal{G}$. Hence the restriction of $\lambda$
to $|\lambda|_I$ defines a Borel measure on $\mathcal{G}$ and hence
it is recognized as a measured lamination $\lambda_{A,B}$ on
$\mathbb{D}$. When we specify the geodesic $I$, we denote
$\lambda_{A,B}$ by $\lambda_I$.

In this notation, if $B$ is a point of $\partial \mathbb{D}$ and
$B\in \partial A$, we recognize $\lambda_I=\lambda_{A,B}$ as the
zero measure. This notation will appear in Proposition
\ref{prop:decay}.

\subsection{The integral is well-defined}
\label{subsec:well-definedness_integral} In this section, we prove
that the integral
\begin{equation} \label{eq:integral}
\int_{\mathcal{G}}\dot{E}^\lambda_\ell (z)d\lambda(\ell)
\end{equation}
is well-defined for all $z\in \partial \mathbb{D}$, when $\lambda\in
\mathcal{ML}_b(\mathbb{D})$.

\begin{remark}
Recall that when we fix $z\in \partial \mathbb{D}$,
$$
\mathcal{G}\ni \ell\mapsto \dot{E}^\lambda_\ell(z)
$$
is a function with the domain $\mathcal{G}$. Notice from the
definition that for $z\in \partial \mathbb{D}$,
$\dot{E}^\lambda_\ell(z)$ is independent of the measure $\lambda$,
depends only on the support $|\lambda|$ of $\lambda$. Hence we can
define $\dot{E}^\lambda_\ell(z)$ for any geodesic lamination
$\lambda$.
\end{remark}

\subsubsection{Support of the integral.}
Let $A$ be the fixed stratum which we used to define
$\dot{E}^\lambda_\ell(z)$ in \S\ref{sec:infinitesimal_earthquake}.
Let $\ell_A$ be the leaf of $\lambda$ contained in the closure of
$A$ which is closest to $z$. Let $z_0$ be a point of $\ell_A$.

Let $I$ be the geodesic connecting $z_0$ and $z$. If $z\in \partial
\mathbb{D}\cap \overline{A}$, $\dot{E}^{\lambda}_{\ell}(z)$ is
identically $0$ on $\mathcal{G}$. Hence the integral
\eqref{eq:integral} converges in this case. Hence we may assume that
$z$ is not in $\overline{A}$. This means that $I\cap A=\{z_0\}$ and
$I$ is not contained in any leaf of $\lambda$.

We define a measured lamination $\lambda_I$ as before. As above, we
denote by $|\lambda|_I$ the support of $\lambda_I$. Namely,
$|\lambda|_I=|\lambda_I|=|\lambda_{A,z}|$.

The following lemma is immediate from the definition of
$\dot{E}^{\lambda}_{\ell}(z)$.

\begin{lemma} \label{lem:support}
Suppose $\lambda$ is a geodesic lamination. Then, for $z\in \partial
\mathbb{D}$, the support of a function $\mathcal{G}\ni\ell\mapsto
\dot{E}^{\lambda}_{\ell}(z)$ is equal to $|\lambda|_I
=|\lambda_{A,z}|$.
\end{lemma}

\subsubsection{A function $\tilde{e}_z$ on $\mathcal{G}$}
For $z\in \partial \mathbb{D}$,
we define a function $\tilde{e}_z$ on $\mathcal{G}$ as follows.
Let $\ell=\geodesic{a}{b}$.
We set
\begin{equation} \label{eq:extension_earthquake}
\tilde{e}_z(\ell):=\left\{
\begin{array}{cl}
\frac{(z-a)(z-b)}{a-b} & \mbox{$a\ne z$ and $b\ne z$} \\
0 & \mbox{otherwise},
\end{array}
\right.
\end{equation}
where in the first row of the right-hand side of
\eqref{eq:extension_earthquake}, $a$ and $b$ are chosen such that
the ordered triple $(a,z,b)$ lies on $\partial \mathbb{D}$
counterclockwise. For instance, in Figure
\ref{fig:geodesics_ell_elldash}, we have
$\tilde{e}_z(\ell)=\frac{(z-a)(z-b)}{a-b}$ and
$\tilde{e}_{z'}(\ell)=\frac{(z'-b)(z'-a)}{b-a}$.
\begin{figure}
\includegraphics[height=4cm]{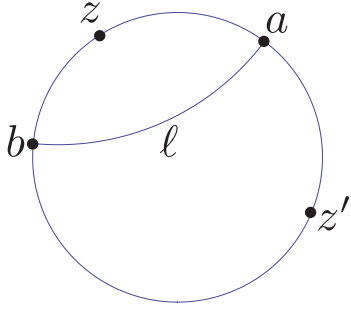}
\caption{Geodesics $\ell$ and $\ell'$.}
\label{fig:geodesics_ell_elldash}
\end{figure}
Notice that $\tilde{e}_z$ is well-defined and continuous on
$\mathcal{G}$. Since $\tilde{e}_z(\ell)=\dot{E}^\lambda_\ell (z)$ on
the support $|\lambda|_I$ of $\lambda_I$, by Lemma
\ref{lem:support}, we conclude the following.

\begin{lemma} \label{lem:measurable}
Let $\lambda$ be a measured lamination. Then, the function
$\mathcal{G}\ni\ell\mapsto \dot{E}^\lambda_\ell (z)$ is measurable
with respect to $\lambda$. Furthermore, for any $z\in \partial
\mathbb{D}$, if the geodesic ray $I$ above is not contained in any
leaf of $\lambda$, it holds
\begin{equation} \label{eq:integral2}
\int_{\mathcal{G}}\dot{E}^\lambda_\ell (z)d\lambda(\ell)=
\int_{\mathcal{G}}\tilde{e}_z(\ell)d\lambda_I(\ell)
=\int_{\mathcal{G}}\tilde{e}_z(\ell)d\lambda_{A,z}(\ell),
\end{equation}
if either the middle term or the right-hand side of \eqref{eq:integral2}
are defined.
\end{lemma}

In particular,
the integral \eqref{eq:integral}
is represented as the integration of a continuous function
defined independently of $\lambda$,
but depending only on $z$.
Thus,
to check
the convergence of the integral \eqref{eq:integral},
we may prove the integrability of $\tilde{e}_z$
with respect to $\lambda_{A,z}$.

We now give properties of the function $\tilde{e}_z$. One can easily
see that
$$
\tilde{e}_{T(z)}(T(\ell))T'(z)^{-1}=\tilde{e}_z(\ell)
$$
for all $\ell\in \mathcal{G}$, $z\in \partial \mathbb{D}$ and $T\in
\mathrm{M\ddot{o}b}(\mathbb{D})$. Let $J$ be the radial geodesic ray
emanating from $0$ to $z\in
\partial \mathbb{D}$. Let $w_d$ ($d\ge 0$) be the length
parametrization of $J$ with $w_0=0$. The function $\tilde{e}_z$ has
the following property.

\begin{lemma} \label{lem:degeneration}
Let $z\in \partial \mathbb{D}$.
For $D_0>0$,
it holds
$$
|\tilde{e}_z(\ell)|\le (8\cosh(D_0))e^{-d}
$$
when $\ell$ intersects the $D_0$-neighborhood of $w_d$.
\end{lemma}

\begin{proof}
Notice that the set $K_0\subset \mathcal{G}$ of all geodesics
intersecting the hyperbolic disk of center $0$ and radius $D_0$ is
compact. By a hyperbolic trigonometry formula, we have
$$
|\tilde{e}_z(\ell)|=|(z-a)(z-b)|/|a-b|\le 4/|a-b|\le 8\cosh(D_0)
$$
for all $\ell =\geodesic{a}{b}\in K_0$ and $z\in \partial
\mathbb{D}$.

Let $\ell$ be a geodesic which intersects the $D_0$-neighborhood of
$w_d$. Let $T$ be a M\"obius transformation acting on $\mathbb{D}$
with $T(w_{d})=0$ and fixing $z$. Since $w_{d}$ is on $J$,
$w_{d}=|w_{d}|z$. Since $T(\ell)\in K_0$, we obtain
\begin{align*}
|\tilde{e}_z(\ell)|
&=
|\tilde{e}_{T(z)}(T(\ell))|
|T'(z)|^{-1}
\le
(8\cosh(D_0))|1-\overline{w_{d}}z|^2/(1-|w_{d}|^2) \\
&=(8\cosh(D_0))\frac{1-|w_{d}|}{1+|w_d|} = (8\cosh(D_0))e^{-d},
\end{align*}
which implies what we wanted.
\end{proof}

\subsubsection{Proof that the integral is well-defined}

%
Recall that $A$ is the stratum which we fixed in the begining and
$z_0\in A$ is the initial point of $I$. Let $z_d$ ($d\ge 0$) be the
length parametrization of $I$. We set $I_d=\{z_k\mid k\ge d\}$. We
can define a measured lamination $\lambda_{I_d}$ as above. Notice
that if $|\lambda|_I$ contains no leaves which diverge in
$\mathcal{G}$, the support $|\lambda|_I$ of $\lambda_I$ is compact
and eventually $\lambda_{I_d}$ becomes the zero measure.

The integral \eqref{eq:integral} for bounded measured laminations
converges because of the following estimate.

\begin{proposition}[Rate of decay] \label{prop:decay}
Let $\lambda\in \mathcal{ML}_b(\mathbb{D})$ and $z\in \partial
\mathbb{D}$. Let $\ell_A$ be the leaf of $\lambda$ in $A$ facing
$z$. Let $z_0\in \ell_A$ and $I$ be the geodesic ray emanating from
$z_0$ and terminating at $z$ as above. Then, there is a constant
$C_2$ depending only on the hyperbolic distance between $0$ and
$z_0$ such that
\begin{equation} \label{eq:decay_1}
\int_{\mathcal{G}} |\tilde{e}_z(\ell)| d\lambda_{I_d}(\ell)\le
C_2\|\lambda\|_{Th}\cdot e^{-d}
\end{equation}
for $d\ge 0$.
\end{proposition}

\begin{proof}
When $z$ is in the closure of $A$, the interval $I$ is contained in
$A$. Hence $\lambda_I$ is the zero measure, and \eqref{eq:decay_1}
holds for all $d\ge 0$. In this case $\dot{E}_{\ell}^\lambda(z)$ is
identically zero on $\mathcal{G}$. Therefore, the integral in
\eqref{eq:integral} converges and equals to zero (and the equation
\eqref{eq:integral2} also holds). Hence we may assume that $z\in
\partial \mathbb{D}\setminus \overline{A}$. This assumption means
that $I$ transversely intersects some leaves of $\lambda$ in
$\mathbb{D}$. However, note that $z$ may be an endpoint of some leaf
of $\lambda$.

Let $\{I_{n,d}\}_{n=0}^\infty$ be a sequence of consecutive
subintervals of $I_d$
such that $z_d\in I_{0,d}$
and
$I_{n,d}\cap I_{n+1,d}=\{z_{d+n}\}$.
Notice that each $I_{n,d}$ has unit length.
We define a measured sublamination $\lambda_{I_{n,d}}$ of $\lambda_I$
as above.
When there is no leaf of $\lambda$ intersecting $I_{n,d}$,
we define $\lambda_{I_{n,d}}$ to be the zero measure as we noted
before.

As in Lemma \ref{lem:degeneration}, we denote by $J$ the radial
geodesic ray emanating from $0$ to $z$, and $w_d$ ($d\ge 0$) the
length parametrization of $J$ with $w_0=0$. Let $\ell$ be a leaf of
$\lambda_{I_{n,d}}$ and $\{z_{d'}\}=\ell\cap I_{n,d}$. Then, by the
triangle inequality, we have $d_{\mathbb{D}}(0,z_{d'})\ge n+d-D_0$.
Since $J$ shares the endpoint $z$ with $I$,
$d_\mathbb{D}(w_{d'},z_{d'})\le d_\mathbb{D}(z_0,w_0)=D_0$, which
means that any leaf of $\lambda_{I_{n,d}}$ intersects the
$D_0+1$-neighborhood of $w_{d'}$. By Lemma \ref{lem:degeneration},
we have
$$
|\tilde{e}_z(\ell)| \le (8\cosh(D_0+1))e^{-d_{\mathbb{D}}(0,z_{d'})}
\le (8\cosh(D_0+1))e^{-(d+n-D_0)}=C_1e^{-(d+n)},
$$
where $C_1=8e^{D_0}\cosh(D_0+1)$.
%
Therefore,
we get
\begin{align*}
\int_{\mathcal{G}}
|\tilde{e}_z(\ell)|d\lambda_{I_{n,d}}(\ell)
&\le C_1e^{-(d+n)}\lambda_{I_{n,d}}(\mathcal{G})
=
C_1e^{-(d+n)}\lambda_{I_{n,d}}(I_{n,d}) \\
&\le C_1\|\lambda\|_{Th}\,e^{-d}\cdot e^{-n},
\end{align*}
since each $I_{n,d}$ has unit length
and the support of $\lambda_{I_{n,d}}$ is contained in $I_{n,d}$.
Thus,
we conclude
$$
\int_{\mathcal{G}}
|\tilde{e}_z(\ell)|
d\lambda_{I_{d}}(\ell)
\le \sum_{n=0}^\infty
\int_{\mathcal{G}}
|\tilde{e}_z(\ell)|
d\lambda_{I_{n,d}}(\ell)
\le C_2\|\lambda\|_{Th}e^{-d},
$$
where $C_2=(1-e^{-1})C_1$.
\end{proof}

\subsection{Weak* convergence and Pointwise convergence}
\label{subsec:weakconv_pointwise} In this section, we prove the
continuity of the integral \eqref{eq:integral_lamination} on
$\mathcal{ML}_b(\mathbb{D})$ with respect to the weak* topology.

\begin{proposition}[Pointwise convergence] \label{thm:pointwise_convergence}
Fix $\alpha$ with $0\le \alpha<1$. Let $\{\lambda_n\}_{n=1}^\infty$
be a sequence of measured laminations which converges in the weak*
topology to a measured lamination $\lambda\in
\mathcal{ML}_b(\mathbb{D})$. If the Thurston norms of the sequence
$\{\lambda_n\}_{n=1}^\infty$ of measured laminations are uniformly
bounded, then there is a choice of normalizations for
$\dot{E}^{\lambda}_{\ell}$ and $\dot{E}^{\lambda_n}_{\ell}$ such
that
$$
\lim_{n\to \infty} \int_{\mathcal{G}} \dot{E}^{\lambda_n}_{\ell}
(z)d\lambda_n(\ell ) = \int_{\mathcal{G}} \dot{E}^{\lambda}_{\ell}
(z)d\lambda (\ell )
$$
for all $z\in \partial\mathbb{D}=S^1$.
\end{proposition}

\begin{proof} The proof follows the same outline as the proof of
\cite[Lemma 3.2]{Saric2}. We first fix the normalizations of
$\dot{E}^{\lambda}_{\ell }$ and $\dot{E}^{\lambda_n}_{\ell}$. Let
$A$ be a fixed stratum of $\lambda$ which is either a gap of
$\lambda$ or a leaf of $\lambda$ whose $\lambda$-measure is zero
(i.e. $A$ is not an atom of $\lambda$). Let $z_0\in A$ be a point in
the interior of $A$ if it is a gap, or any point of $A$ if it is a
leaf of $\lambda$. Let $A_n$ be the stratum of $\lambda_n$ which
contains $z_0$. We orient each $\ell\in |\lambda |$ to the left as
seen from $A$. If $A$ is a geodesic, then we orient $A$ arbitrary.
This gives a well-defined function $\dot{E}^{\lambda}_{\ell }$ for
$\ell\in |\lambda |$ which in turn implies
$$
\int_{\mathcal{G}}\dot{E}^{\lambda}_{\ell}(z)d\lambda (\ell
)=\int_{\mathcal{G}}\tilde{e}(\ell )d\lambda_{A,z}(\ell ).
$$
We define $\dot{E}^{\lambda_n}_{\ell}$ by giving the left
orientation to each $\ell$ with respect to the stratum $A_n$ in the
same fashion.

Let $I$ be a geodesic ray from $z_0$ to $z$ and let $z_d\in I$ be
such that the distance between $z_0$ and $z_d$ is $d\geq 0$. We fix
$d>0$ such that $z_d$ is contained in a stratum $A_{d}$ of $|\lambda
|$ which is either a gap or a leaf which is not an atom of
$\lambda$.

Given $i\in\mathbb{N}$, let $I_i=(z_l^i,z_r^i)$ be an open geodesic
arc whose endpoints are on the distance $1/i$ from $z_0$ and $z_d$,
and which contains $z_0,z_d$. The set of geodesics of $\mathbb{D}$
which intersect $I_i$ is open in $\mathcal{G}$ and contains all
geodesics of $|\lambda |$ which intersect the closed geodesic arc
with endpoints $z_0$ and $z_d$. Since the lengths of $(z_l^i,z_0)$
and $(z_d,z_r^i)$ are going to zero as $i\to\infty$, it follows that
the $\lambda$-measure of the set of geodesics intersecting
$(z_l^i,z_0)$ and $(z_d,z_r^i)$ is going to zero as $i\to\infty$ by
the choice of $z_0$ and $z_d$ (namely, $A$ and $A_{z_d}$ are either
gaps or non-atomic leaves). Let $\varphi_i:\mathcal{G}\to\mathbb{R}$
be a non-negative continuous function whose support consists of
geodesics intersecting $I_i=(z_l^i,z_r^i)$ and which is identically
equal to $1$ on the set of geodesics intersecting $[z_0,z_d]$. Then
the function $\ell\mapsto \varphi_i(\ell )\tilde{e}_{\ell}(z)$ is a
continuous function on $\mathcal{G}$ with compact support. It
follows that
$$
\int_{\mathcal{G}}\varphi_i(\ell )\tilde{e}_{\ell}(z)d\lambda_n(\ell
)\to \int_{\mathcal{G}}\varphi_i(\ell )\tilde{e}_{\ell}(z)d\lambda
(\ell )
$$
as $n\to\infty$ by the weak* convergence $\lambda_n\to\lambda$.

Note that
$$\int_{\mathcal{G}}\varphi_i(\ell
)\tilde{e}_{\ell}(z)d\lambda_n(\ell )\leq
\int_{\mathcal{G}}|\tilde{e}_{\ell}(z)|d[(\lambda_n)_{(z_l^i,z_0)}+(\lambda_n)_{(z_d,z_r^i)}](\ell
)+\int_{\mathcal{G}}\tilde{e}_{\ell}(z)d(\lambda_n)_{(z_0,z_d)}(\ell
)
$$
and
$$\int_{\mathcal{G}}\varphi_i(\ell
)\tilde{e}_{\ell}(z)d\lambda (\ell )\leq
\int_{\mathcal{G}}|\tilde{e}_{\ell}(z)|d[\lambda_{(z_l^i,z_0)}+\lambda_{(z_d,z_r^i)}](\ell
)+\int_{\mathcal{G}}\tilde{e}_{\ell}(z)d\lambda_{(z_0,z_d)}(\ell ).
$$

The choice of $z_0$ and $z_d$ is such that the total masses of
$\lambda_{(z_l^i,z_0)}$ and $\lambda_{(z_d,z_r^i)}$ on $\mathcal{G}$
converge to zero as $i\to\infty$. Since $\lambda_n$ converges to
$\lambda$ in the weak* sense, it follows that given $\epsilon >0$
there exist $i_0,n_0\in\mathbb{N}$ such that the total masses of
$\lambda_{(z_l^i,z_0)}$, $\lambda_{(z_d,z_r^i)}$,
$(\lambda_n)_{(z_l^i,z_0)}$ and $(\lambda_n)_{(z_d,z_r^i)}$ on
$\mathcal{G}$ are less than $\epsilon$ for $i\geq i_0$ and $n\geq
n_0$. The above three inequalities imply that
$$
\int_{\mathcal{G}}\tilde{e}_{\ell}(z)d(\lambda_n)_{(z_0,z_d)}(\ell
)\to \int_{\mathcal{G}}\tilde{e}_{\ell}(z)d\lambda_{(z_0,z_d)}(\ell
)
$$
as $n\to\infty$.

Since
$|\int_{\mathcal{G}}\tilde{e}_{\ell}(z)d(\lambda_n)_{(z_0,z_d)}(\ell
)-\int_{\mathcal{G}}\tilde{e}_{\ell}(z)d\lambda_n(\ell )|\leq
Ce^{-d}$ and
$|\int_{\mathcal{G}}\tilde{e}_{\ell}(z)d\lambda_{(z_0,z_d)}(\ell
)-\int_{\mathcal{G}}\tilde{e}_{\ell}(z)d\lambda (\ell )|\leq
Ce^{-d}$, the conclusion follows.
\end{proof}

\subsection{Differentiation of earthquake paths}
\label{subsec:differential_formula}
In this section,
we reprove the formula \eqref{eq:equation_E_lambda}.

\subsubsection{Holomorphic motions and Complex earthquakes}
\label{subsubsec:holmorphic_motion_complex_earthquake} Let $S$ be a
subset of $\hat{\mathbb{C}}$ and let $D$ be a domain in
$\hat{\mathbb{C}}$. A \emph{holomorphic motion of $S$ over $D$ with
base point $t_0\in D$} is, by definition, a map $h:S\times D\to
\hat{\mathbb{C}}$ satisfying the following three properties:
\begin{itemize}
\item[(1)]
$h(x,t_0)=x$ for all $x\in S$.
\item[(2)]
For all $t\in D$,
$h_t(\cdot):=h(\cdot, t)$ is injective on $S$.
\item[(3)]
For all $s\in S$, $h(s,\cdot ):D\to \hat{\mathbb{C}}$ is holomorphic.
\end{itemize}
By Slodkowski's theorem (\cite{Slod}),
if $D$ is conformally equivalent to the unit disk,
any holomorphic motion $h$ of $S$ over $D$
with base point $t_0\in D$
extends to a holomorphic motion $\tilde{h}$
of $\hat{\mathbb{C}}$ over $D$
and for each $t\in D$,
$\tilde{h}_t$ is $K_t$-quasiconformal mapping
where $K_t=\exp(d_D(t_0,t))$
and $d_D$ is the Poincar\'e distance on $D$
 normalized such that it has curvature $-1$.

The following theorem is proved in \cite{Saric4}.

\begin{theorem}[Theorem 2 in \cite{Saric4}] \label{thm:holomorphic_motion_saric}
Let $\lambda\in ML_b(\mathbb{D})$. The earthquake map $ (z,t)\mapsto
E^{t\lambda}(z)$ for $t>0$ and $z\in \partial\mathbb{D}$ extends to
a holomorphic motion $ (z,\tau)\mapsto E^{\tau\lambda}(z)$ of
$\partial \mathbb{D}$ over a neighborhood $S_\lambda$ of
$\mathbb{R}$ in $\mathbb{C}$ with base point $\tau=0$.
\end{theorem}
The domain $S_\lambda$ in the theorem above is
concretely defined by
\begin{equation} \label{eq:domain_S_lambda}
S_\lambda=\{\tau=t+is\mid
|s|<\epsilon_0/[C_0\exp(\|t\lambda\|_{Th})\|\lambda\|_{Th}]\},
\end{equation}
where $\epsilon_0$ and $C_0$ are independent of $\lambda$.




\begin{proof}[Proof of Proposition \ref{prop:convergence_of_integral}]
We first show the convergence in the case when
$\{\lambda_n\}_{n=1}^\infty$ is a finite approximation of $\lambda$.
From the proof of Theorem 2 in \cite{Saric4}, we know that there is
a neighborhood $V_0$ of $\partial \mathbb{D}$ such that the
complement of $V_0$ contains at least $3$ points and
$E^{\tau\lambda_n}(z)\in V_0$ for all $\tau\in S_\lambda$, $z\in
\partial \mathbb{D}$ and $n\in \mathbb{N}$, where we assume in the
definition that the restriction of $E^{t\lambda_n}$ is the identity
on a stratum of $\lambda_n$ containing $A$. This implies that
$\{E^{\tau\lambda_n}(z)\}_{\tau\in S_\lambda}$ is normal family and
converges to $E^{\tau\lambda}(z)$ on any compact set of $S_\lambda$.
From the Weierstrass' theorem, we have
$$
\left.\frac{d}{d \tau}E^{\tau\lambda}(z) \right|_{\tau=0} =
\lim_{n\to 0}\left.\frac{d}{d \tau}E^{\tau\lambda_n}(z)
\right|_{\tau=0}.
$$
On the other hand, by Theorem \ref{thm:pointwise_convergence}, the
integral in \eqref{eq:integral_lamination} varies continuously on
$\mathcal{ML}_b(\mathbb{D})$. Hence, we get the formula
(\ref{eq:equation_E_lambda}).
\end{proof}

\end{document}